\documentclass[a4paper,11pt]{amsart}

\usepackage{graphicx}

\usepackage{amssymb}
\usepackage{amsfonts}
\usepackage{indentfirst}
\usepackage{graphicx, graphics}
\usepackage{pst-plot}

\def\epsi{\varepsilon}
\def\mup{{\mu_\phi}}

\def\N{\mathbb{N}}
\def\rr{\mathbb{R}}

\def\Lkap{\mathcal{L}^\delta(x)}
\def\Fkap{\mathcal{F}^\delta(x)}
\def\C{\mathcal{C}}

\def\e{\varepsilon}

\def\muphi{\mu_{\phi}}
\def\un\mathbb{1}
\def\zu{[0,1]}
\def\mk{\medskip}
\def\sk{\smallskip}

\newtheorem{theorem}{Theorem}[section]
\newtheorem{lemma}[theorem]{Lemma}
\newtheorem{proposition}[theorem]{Proposition}
\newtheorem{corollary}[theorem]{Corollary}
\newtheorem{remark}[theorem]{Remark}
\newtheorem{definition}[theorem]{Definition}

\begin{document}

\title[Diophantine approximation by  orbits of Markov maps]{Diophantine approximation \\ by orbits of expanding Markov maps}

\author{Lingmin Liao}
\address{LAMA,  CNRS UMR 8050,
Universit\'e Paris-Est   Cr\'eteil, 61 Avenue du
G\'en\'eral de Gaulle, 94010 Cr\'eteil Cedex, France}
\email{lingmin.liao@u-pec.fr}

\author{St\'ephane Seuret}
\address{LAMA, CNRS UMR 8050,
Universit\'e Paris-Est  Cr\'eteil, 61 Avenue du
G\'en\'eral de Gaulle, 94010 Cr\'eteil Cedex, France}
\email{seuret@u-pec.fr}

\begin{abstract} In 1995, Hill and Velani   introduced the ``shrinking targets" theory. Given a dynamical system $(\zu,T)$, they investigated   the Hausdorff dimension  of sets of points whose orbits are close to some fixed point. In this paper, we   study the sets of points well-approximated by orbits $\{T^n
x\}_{n\geq 0}$, where $T$ is an expanding Markov map with a finite
partition supported by  $\zu$. The  
dimensions of these sets are described using the multifractal properties of 
invariant Gibbs measures.  \end{abstract}

\maketitle

\section{Introduction}

Let $(X, d)$ be a compact metric space and $T: X \rightarrow X $ a
piecewise continuous  transformation.   Let
$\mathcal{O}(x)= \{ T^n x: n\in \N\}$ be the orbit of $x\in X$. The distribution of $\mathcal{O}(x)$, in particular its density over $X$, is a historical issue,
which goes back   to   Poincar\'e's results.
In 1995, Hill and Velani \cite{HV} introduced the shrinking targets theory, which aims at investigating the Hausdorff
dimensions of sets of points whose orbits are close to some
fixed point. For a   point $y\in X$, they studied the set
\begin{equation}\label{set-HV}
   \Big\{ x\in X : \quad T^nx\in B(y, r_n) \quad \text{for infinitely many integers } n\in \mathbb{N} \
   \Big\},
\end{equation}
where $B(x, r)$ stands for the ball of radius $r>0$ centered at $x
\in X$ and $(r_n)_{n\geq 1}$ is a sequence of positive real numbers
converging to $0$. In this article, we adopt a complementary point of view: we
fix a point $x\in X$ and consider the set of points $y$
well-approximated by the orbit $\mathcal{O}(x)$ of $x$, i.e. we
focus on  
\begin{equation}\label{set-FST}
   \Big\{ y\in X :\quad T^nx\in B(y, r_n) \quad \text{for infinitely many integers } n\in \mathbb{N} \
   \Big\},
\end{equation}
which can be written as $ \displaystyle  \limsup_{n\to \infty} B(T^n x,
r_n)=\bigcap_{N\geq 1} \ \bigcup_{n\geq N} B(T^n x, r_n).$

%


In fact, many questions can be asked about the set 
(\ref{set-FST}): for which sequence $(r_n)_{n\geq 1}$
does it cover  the whole interval $\zu$? When $\zu$ is not fully
covered, what is its Hausdorff dimension? Can  the dependence on $x$ be quantified?
Answering these questions provides us with a  precise description of  the distribution properties of the orbit $ \mathcal{O}(x)$.
Such questions have been investigated in several contexts, and can
be interpreted as general Diophantine approximation problems.
Indeed, the classical Diophantine questions concern the dimension of
the set \begin{equation}  \label{defsdelta}  \mathcal{S}(\delta) =
\left\{y\in \zu:  \left|y-\frac{p}{q} \right| \leq
\frac{1}{q^{2\delta}} \mbox{ for infinitely many couples
$(p,q)=1$}\right\}.
\end{equation}
This set  can again be seen as a limsup set  $\displaystyle \limsup_{q\to +\infty} \ \bigcup_{p\in \mathbb{Z}} B(p/q, 1/q^{2\delta})$.

The work \cite{HV} is precursor on this subject in the dynamical
setting, and thereafter, many people studied  sets of the form
(\ref{set-HV}). For instance, see  \cite{KIM} for the case where $T$
is an irrational rotation on the torus $\mathbb{T}^1$. In the
literature, one often refers  to these results as shrinking targets problems  or
dynamical Borel-Cantelli lemma. The paper \cite{FST} by Fan,
Schmeling and Troubetzkoy, where the doubling map on $\mathbb{T}^1$
is studied, is the first one to consider the set (\ref{set-FST}).
These studies are also related to many other famous works concerned
with metric theory of Diophantine approximation, see
\cite{JWS,DODVEL,HV2,KM,KM2,BV1,EKL} and references therein.

In this work, we focus on the study of the set (\ref{set-FST}) when
$T$ is an expanding Markov map of the interval $\zu$ with a finite
partition - Markov map, for short.

\begin{definition}[Markov map]\label{defmarkov}
A transformation $T: [0,1] \rightarrow [0,1]$ is an expanding Markov
map with a finite partition if there is a subdivision $
\{a_i\}_{0\leq i\leq Q} $ of $[0,1]$ (denoted by
$I(k)=]a_k,a_{k+1}[$ for $0\leq k \leq Q-1 $) such that:
\begin{enumerate}
\item  there is an integer $n$ and a real number $\rho$ such that $|(T^n)'| \geq \rho
>1$;
\item   $T$ is strictly monotonic and can be extended to a $C^2$
function on each $\overline{I(i)}$;
\item   if $I(j) \cap T(I(k)) \neq \emptyset $, then $I(j) \subset
T(I(k))$;
\item  there is an integer $R$ such that $I(j) \subset \cup_{n=1}^{R} T^n(I(k))
$ for every $k,j$; and
\item  for every $k\in\{0,\cdots, Q-1\}$,  $\sup_{ (x,y,z)\in I(k)^3}\frac{|T''(x)|}{|T'(y)||T'(z)|}<
\infty.$
\end{enumerate}
\end{definition}

 It appears that for Markov maps,
the relevant choice for the sequence $(r_n)_{n\geq 1}$ is $r_n =
1/n^\delta$, for $\delta>0$. We thus introduce the sets
\begin{eqnarray*}
     \mathcal{L}^{\delta}(x) & :
     =  &\limsup_{n\to \infty} B(T^n x,       n^{-\delta}),\\
           \mathcal{F}^{\delta}(x)  &:      = & \zu \setminus \mathcal{L}^{\delta}(x),
\end{eqnarray*}
which are   the set of points covered by infinitely many
balls $B(T^n x,  n^{-\delta})$, and its complement. We study  the  Hausdorff dimension of  $\mathcal{L}^{\delta}(x)$ and  $
\mathcal{F}^{\delta}(x)$.

Before stating our main theorem, we give some recalls on multifractal analysis.
 Denote by $Leb$ the  Lebesgue measure in $\mathbb{R}$, and by $\dim $ the Hausdorff dimension.
\begin{definition}
For any Borel  probability measure $\mu$ on $\zu$, define the lower  ({\em resp.} upper) local dimension $\underline{ d}_\mu(y)$ ({\em resp.} $\overline{ d}_\mu(y)$) of  $\mu $ at $y\in \zu$ by
$$
\underline{ d}_\mu(y) := \liminf_{r\to 0} \frac{\log \mu(B(y,r))}{\log r} \ \mbox { and } \  \overline{ d}_\mu(y) := \limsup_{r\to 0} \frac{\log \mu(B(y,r))}{\log r}.$$
When $\underline{ d}_\mu(y) =  \overline{ d}_\mu(y) $, their common value is denoted by $ d_\mu(y)$, and is simply called the local dimension of $\mu$ at $y$.
The level sets of the local dimension are
\begin{equation}
\label{defemu}
 {\mathcal{E}}_{\mu} (\alpha) = \left\{y \in \zu: d_\mu(y)=\alpha
\right\}, \quad \alpha\geq 0.
\end{equation}
Finally, the {\em multifractal spectrum} of $\mu$, is defined as  the application
  $$ D_\mu : \alpha \geq 0 \ \longmapsto \ \dim  \,  {\mathcal{E}}_{\mu} (\alpha).$$
\end{definition}

Denote by $\mathcal{M}_{\rm inv} $ ({\em resp.} $\mathcal{M}_{\rm erg} $ ) the set of $T$-invariant  ({\em resp.} ergodic) 
probability measures on $\zu$. The dimension of a Borel probability measure $\mu$ is
defined as
$$\dim  \mu = \inf \{ \dim  E: E \mbox { Borel set $\subset \zu$ and } \mu(E)>0\} .  $$

We  prove the following theorem (for precise definitions, see Section \ref{sec:def}).

\begin{theorem}\label{main-thm1}
Let $T:\zu\to \zu$ be an expanding Markov map.   Let $\phi$ be a
H\"{o}lder continuous potential  and let $\mu_{\phi}$ be the
corresponding Gibbs measure.  Let
\begin{equation}
\label{defalphamax}  \alpha_{+}:= \max_{\mu\in \mathcal{M}_{\rm inv}}
{\int_{\zu}   (-\phi) \, d\mu \over \int_{\zu}    \log |T'|  \,  d\mu} \ \mbox { and } \ \alpha_{\max}:= {  {\int_{\zu}     (-\phi) \, d
\mu_{\max}}  \over \int_{\zu}    \log |T'|  \, d{\mu_{\max}}},
\end{equation}
where $\mu_{\max}$ is the Gibbs measure associated with the
potential $\psi=-\log|T'|$.

\begin{enumerate}
\item  For $\mu_{\phi}$-almost every $x \in
\zu$,  the Hausdorff dimension of $\mathcal{L}^\delta(x)$ is
\begin{equation}
\label{resultat}   \dim  \mathcal{L}^\delta(x)  = \left\{ \begin{array}{ll}
                           {1}/{\delta} & \mbox{\rm if} \ \ 0< {1}/{\delta}\le \dim {\mu_\phi},\sk  
                                \\
                          D_\mup( {1}/{\delta}) & \mbox{\rm if} \ \
                           \dim {\mu_\phi}< {1}/{\delta}\leq
                           \alpha_{\max}, \sk  \\
                          1& \mbox{\rm if} \ \   {1}/{\delta}   >  \alpha_{\max}.
                                \\
                          \end{array} \right.
\end{equation}
\item   For $\mu_{\phi}$-almost every $x \in
\zu$,  the Hausdorff dimension of $\mathcal{F}^\delta(x)$ is
\begin{equation}
\label{resultat2}    \dim  \mathcal{F}^\delta(x)  = \left\{ \begin{array}{ll}
                        1  & \mbox{\rm if} \ \ 0<{1}/{\delta} \leq  \alpha_{\max}, \sk \\
                           D_\mup({1}/{\delta})   & \mbox{\rm if} \ \  {1}/{\delta}   >  \alpha_{\max}. 
                          \end{array} \right.
\end{equation}
 \item
Concerning the Lebesgue measure of $\mathcal{L}^\delta(x) $ and $\mathcal{F}^\delta(x) $, we have:
 \begin{equation}
\label{resultat3} \!\!\!\!\!\!\!  \mbox{for $\mu_{\phi}$-a.e. $x$, }   \ Leb(\mathcal{L}^\delta(x)) =1- Leb(\mathcal{F}^\delta(x))= \left\{ \begin{array}{ll}
                           0 & \mbox{\rm if} \ \ 0<  {1}/{\delta}  <
                           \alpha_{\max}, \sk\\
                          1& \mbox{\rm if} \ \  {1}/{\delta}  >  \alpha_{\max}.                              \\
                          \end{array} \right.
\end{equation}

\item If ${1}/{\delta}> \alpha_{+}$, then
$\mathcal{F}^\delta(x)=\emptyset$ and hence $\Lkap=\zu$.

\end{enumerate}
\end{theorem}


\begin{remark}
The dimensions of $\mathcal{L}^{1/\alpha_{\max}} (x)$ and   $\mathcal{F}^{1/\alpha_{\max}}(x)$ are $1$, but we do not know
 their Lebesgue measure. We also  prove that
  Leb($\mathcal{L}^{1/\alpha_{+}} (x))=1$  and $\dim  \mathcal{F}^{1/\alpha_{+}} (x) \leq \lim_{1/\delta \to
  \alpha_+} D_\mup(1/\delta)$, but we do not know if  $\mathcal{L}^{1/\alpha_{+}}(x)=\zu$.
\end{remark}

%

%
%
%
%
%

 The mapping $ 1/\delta \longmapsto \dim  \mathcal{L}^\delta(x)$ exhibits clearly four distinct behaviors, respectively denoted by  Part I (${1}/{\delta}\le \dim {\mu_\phi}$), Part II ($\dim {\mu_\phi} < {1}/{\delta} \le \alpha_{\max}$), Part III ($\alpha_{\max} <{1}/{\delta} \le \alpha_{+}$) and finally Part IV (${1}/{\delta}> \alpha_{+}$). See Figure 1 (the definition of $\alpha_-$ is given in Theorem \ref{th_recall}).
 
 \begin{figure}
\begin{center}
\includegraphics[width=5.2cm,height =4.3cm]{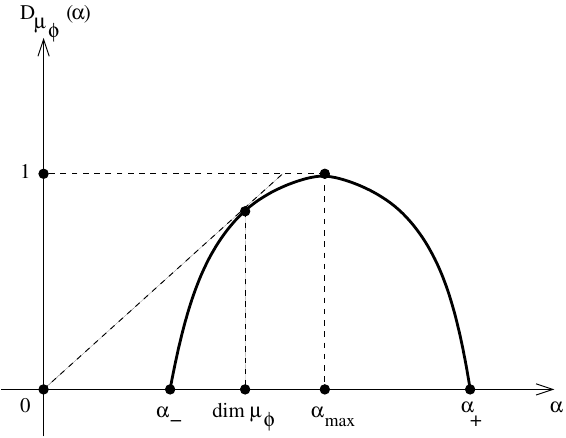}\\
\includegraphics[width=5.2cm,height =4.3cm]{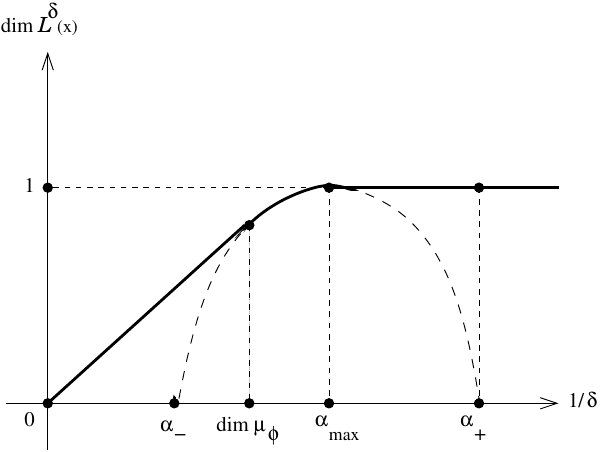}
\includegraphics[width=5.2cm,height =4.3cm]{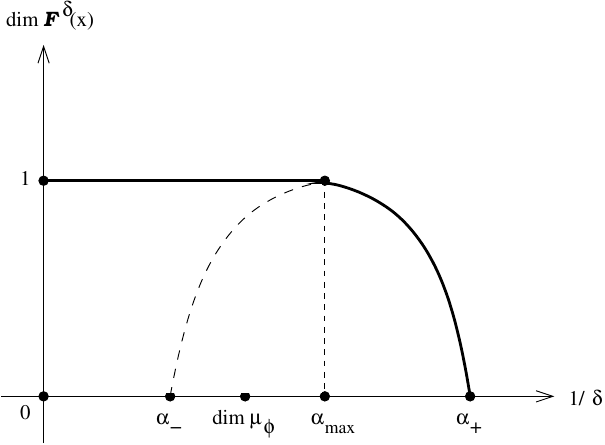}
\caption{Multifractal spectrum of  $\mup$,  and the  maps  $1/\delta \mapsto     \dim  \,\Lkap$ and $1/\delta \mapsto    \dim    \Fkap$.}\end{center} 
\vskip -8pt
\end{figure} 
 
  \sk

For $\mup$-typical $x$,  the behavior  of $ \mathcal{L}^\delta(x)$ enjoys two remarkable characteristics   compared to classical Diophantine approximation: 
  the map  $1/\delta\mapsto \dim  \mathcal{L}^\delta(x)$  may have a strictly concave part (Part II), and 
   the  smallest $\delta$ for which  $Leb \, ( \mathcal{L}^\delta(x))=1  $ and  the  smallest $\delta$ for which $  \mathcal{L}^\delta(x)= \zu$  do not coincide (Part III). 
This   contrasts with the classical Diophantine
approximation, especially the approximation by rationals. In this historical context, as said above, the  analog of
the sets $\mathcal{L}^\delta(x)$ are the sets $S(\delta)$ defined by
\eqref{defsdelta}. The Dirichlet theorem ensures that
$\mathcal{S}(1) = \zu$, and it is well-known \cite{JARNIK,BESIC}
that for every $\delta>1$, $ \dim  \mathcal{S}(\delta)
=1/\delta$. In particular, the dimension  of $\mathcal{S}(\delta)$
decreases linearly with respect to $1/\delta$, and as soon as the
Lebesgue measure of $\mathcal{S}(\delta) $ equals 1, it
instantaneously covers the whole interval $\zu$.  Comparable results
hold for sets of numbers approximated by other families, see for
instance \cite{BS,BS2,BUGEAUD,Bu,SchTro}. 
The two characteristics of  
$\mathcal{L}^\delta(x)$ mentioned above can  be interpreted by the
fact that, although the orbits  $\mathcal{O}(x)$ of
$\mup$-typical points $x$ are dense, they are not as regularly
distributed when $n$ tends to infinity as the rational numbers are.
The exponents $\dim \mup$, $\alpha_{\max}$ and $\alpha_+$
characterize this ``distortion".




 \mk

The paper is organized as follows. Section \ref{sec:def} contains some recalls on multifractal analysis and  hitting time. Section \ref{sec:Hitting} describes the relations between
  hitting time,  approximation rate and  
local dimension of $\mup$. From these relations we will give a
direct proof of item 3. of Theorem \ref{main-thm1}. In Section
\ref{sechitt}, two key lemmas  are proved. They illustrate the fact
that  intervals which have a small local
dimension for $\mu_\phi$  are hit by the balls $B(T^n x, 1/n^\delta)$ with big
probability, and  {\em vice-versa}. Then, Sections \ref{sec-lower1},
\ref{secpart4}, \ref{secpart3} and \ref{secpart2} contains the
proofs of the upper and lower bounds for the Hausdorff dimensions of
$\mathcal{L}^\delta(x)$ and $\mathcal{F}^\delta(x)$ for Parts I, IV,
III and II, respectively.

\sk

 
Theorem \ref{main-thm1} is similar to the results    in \cite{FST} for the doubling map on
$\mathbb{T}^1$, but the proofs require other arguments. First, for $x\mapsto 2x$, since  the Lyapunov exponents are constant,  the balls of generation $n$ have   same lengths,  while for the
Markov maps their lengths may have different order.  Second, in
\cite{FST}, the authors focus on the Bowen's  entropy
spectrum.  Third, the notions of   H\"older exponent and   hitting time   for the
doubling map involve only cylinders, while we need   centered balls
to define similar quantities. 
Finally, some arguments (Lemmas \ref{th-multi-relation}, \ref{Big-hitting}
and \ref{Small hitting}) are adapted from   \cite{FST} to the
context of Markov maps, but several others do not. The
best example  is the difficult lower bound for  $ \dim 
\mathcal{L}^\delta(x)$ when $ 0< {1}/{\delta}\le
\dim {\mu_\phi}$.

\section{First definitions, and recalls on multifractal analysis} \label{sec:def}

%


\subsection{Covering of $\zu$ by basic intervals.}\label{coveringsection} 

Let $T$ be a Markov map defined in Definition \ref{defmarkov}.
With $T$ are associated generations of {\em
basic} intervals.
\begin{definition}
Let $\mathcal{A} = \{0,1,\cdots, Q-1\}$. For every integer $n\geq
1$, we denote by $\mathcal{G}_n$ the set of {\em basic intervals} of
generation $n$ defined by
$$ \mbox{if } (i_1i_2 \cdots i_n)\in\mathcal{A}^n, \ \ I_{i_1i_2\cdots i_n} := I(i_1)\cap T^{-1}(I(i_2)) \cap \cdots \cap  T^{-n+1}(I(i_n)).$$
\end{definition}

The following distortion property between intervals will be crucial:
there is a constant $L>1$ such that for every integer $n\geq 2$, for
every $(i_1i_2 \cdots i_n )\in\mathcal{A}^{n }$,
\begin{equation}
\label{distor} 1 \leq \frac{| I_{i_1i_2\cdots i_{n-1}} |}{ |
I_{i_1i_2\cdots i_{n-1} i_n} |} \leq L  \ \  \  \  \  \  (|I|\mbox{ is the length of the interval $I$).}
\end{equation}

It is obvious that the intervals   of a given
generation $n$ form a covering of $\zu$. This covering of $\zu$ is
not composed of intervals of the same length.
But using \eqref{distor}, for every real number $ 0<r<1$, one easily
shows that there is a   family of basic intervals $J_1, J_2,
\cdots, J_N$ (not belonging to the same $\mathcal{G}_n$'s) such that:
\begin{itemize}
\vspace{-2mm}\item
$\bigcup_{j=1}^N J_j = \zu$
\vspace{-3mm}\item
for every $j\neq j'$, the intersection of interiors
$\overset{\circ}{J_j} \cap \overset{\circ}{J_{j'}} $ is empty,
\vspace{-2mm}\item
for the same constant $L$ as in \eqref{distor}, 
 for every $j\in
\{1,2,\cdots, N\}$,
\begin{equation}
\label{covering} L^{-1}  r \leq |J_j| \leq L  r.
\end{equation}
\end{itemize}
\vspace{-2mm} For every $n\in
\mathbb{N}$, we fix $\mathcal{C}_n$ one possible   collection of  intervals such that (\ref{covering}) holds for
$r=2^{-n}$. From these considerations, one deduces  that there exists a number $L'>1 $ such
that for all $n\in\N$, the generation $n_{J}$ of a basic interval $J\in \C_n$
satisfies
\begin{equation}
\label{covering1} (L')^{-1}n\leq n_J \leq L'n.\end{equation}
%

\subsection{Definition of hitting time.}  Denote $\mathcal{O}^+(x) :=\mathcal{O}(x)
\setminus\{x\}.$

\begin{definition}
For every $(x,y)\in \zu^2$ and $r>0$, we define the hitting time (first
entrance time) of the orbit of $x$ into the ball $B(y,r)$ by
$$ \tau_r(x,y):= \inf\{n \ge 1: T^nx \in B(y,r)\}.$$
Then we set
\begin{equation}\label{R} R(x,y) := \liminf_{r \to 0} \frac{\log
\tau_r(x,y)}{-\log r}.
\end{equation}
\end{definition}
By convention, when $\mathcal{O}^+(x)\cap B(y,r)=\emptyset$,  we set $ \tau_r(x,y) = +\infty$ and  $R(x,y)
=+\infty$. 

\sk

We  define the hitting time $\tau(x,C)$ of a basic interval $C$ by a point
$x$. Let $m\geq 1$ and let $C\in \mathcal{G}_m$. If $\mathcal{O}(x) \cap C = \emptyset$,  we set $\tau(x,C)=+\infty$. Otherwise, 
\begin{equation}
\label{deftauxc} \tau(x,C):= \inf \{n\geq 0: \, T^n x \in C\}.
\end{equation}

\begin{definition}
\label{defsets}
Let $s\geq 0$ be a real number. We define the sets
\begin{eqnarray}
\label{def2}
  \!   \mathcal{R}_{\ge s}(x)   =   \{y \in [0,1]:   R(x,y) \ge s
     \}   , \ \    \mathcal{R}_{\le s}(x)   = \{y \in [0,1]:   R(x,y) \le s \}, \ \\
\nonumber    \!\!  \mathcal{R}_{> s}(x)   =   \{y \in [0,1]:   R(x,y) > s
     \}   , \ \     \mathcal{R}_{< s}(x)   = \{y \in [0,1]:   R(x,y) < s \}. \ 
\end{eqnarray}
\end{definition}

\subsection{Multifractality of Gibbs measures.} Here are some facts on Gibbs measures.
 \label{multi} 

\begin{theorem}  [\cite{Bowen1975,Walters1978}]
  Let $T: I\rightarrow I$ be a Markov map. Then for any
  H\"{o}lder continuous function $\phi : I \rightarrow \mathbb{R} $,
  there exists a unique equilibrium state $\mu_\phi$ which
  satisfies the following
  Gibbs property: there exist constants $\gamma>0 $ and $P(\phi) $
(the topological pressure associated with $\phi$), such that 
\begin{eqnarray}\label{Gibbs}
\mbox{for
  any basic interval $I  \in \mathcal{G}_n$, }  \ \  \gamma^{-1}\leq \frac{\mu_\phi(I )}{e^{S_n\phi(x)-nP(\phi)}}\leq
  \gamma, \quad \forall x\in I ,
\end{eqnarray}
where $S_n\phi(x)=\phi(x)+\cdots + \phi(T^{n-1}x) $ is the $n$-th Birkhoff sum of $\phi$ at $x$.
\end{theorem}
A potential $\phi$ is often assumed to be {\it normalized}, i.e. $P(\phi)=0$. If it is not the case, we can replace $\phi$ by $\phi-P(\phi)$.

Now, let us recall some standard facts on multifractal
analysis of    Borel measures.
By an extensive literature   \cite{CLP,Rand,BMP,Simpelaere,BPS,PesinWeiss1997Chaos}, the multifractal
analysis of $\mup$ can be achieved, i.e. the multifractal spectrum $D_{\mup}$
of $\mup$ can be computed.   


\begin{theorem}
\label{th_recall} Consider a Markov map $T: [0,1]\rightarrow [0,1]$  and a normalized
H\"{o}lder continuous potential  $\phi: [0,1] \rightarrow
\mathbb{R} $.
\begin{enumerate}
\item
The multifractal spectrum $D_\mup$ of $\mup$ is a
concave analytic map on the interval $]\alpha_-,\alpha_+[$, where
$\displaystyle \alpha_{-}:= \min_{\mu\in \mathcal{M}_{\rm inv}} {\int_{\zu}   (-\phi) \, d\mu \over \int _{\zu}   \log |T'|
 \, d\mu} $ and $\alpha_+$ was defined in \eqref{defalphamax}.
\item
The spectrum $D_\mup$ reaches its maximum value 1 at  $\alpha_{\max}$ defined in \eqref{defalphamax}.
\item
The graph of $D_\mup$ and the first bisector intersect at
a unique point which is  $(\dim  \mup,\dim  \mup)$. Moreover,
$\dim  \mu_{\phi} $ satisfies $$\dim  \mu_{\phi} =  {\int_{\zu}    (-\phi) \, d{\rm \mu_\phi} \over \int_{\zu}   \log |T'|  \, d{\rm \mu_\phi}}.$$
\end{enumerate}
\end{theorem}


%


For every $q\in \rr$, there is a unique real number  $\eta_\phi(q)$
such that the topological pressure  $P(-\eta_\phi(q)\log |T'| + q
\phi)  $ associated with the H\"older potential
$\phi_q:=-\eta_\phi(q)\log |T'| + q \phi$ equals 0. Such a number
exists since the map $P:t \mapsto (-t \log |T'| + q \phi) $ is real-analytic  and
decreasing in $t$. The resulting function $q\mapsto \eta_\phi(q)$ is
real-analytic and concave. We denote by
\begin{equation}
\label{defmuq}
\mu_q:=\mu_{\phi_q}, \ \ \mbox{ where $\phi_q:=-\eta_\phi(q)\log |T'| + q \phi$},
\end{equation}
 the Gibbs measure associated with the potential $\phi_q$.
 Observe that $ \eta_\phi(0)=1$, and $ \eta_\phi(1)=0$. The measures
 $\mu_0$ and $\mu_1$($=\mup$) are associated with the potentials $\phi_0=-\log
 |T'|$ and $\phi_1=\phi$ respectively.
 By a  folklore theorem, the Lebesgue measure $Leb$ is equivalent to   $\mu_0 $, coinciding with $\mu_{\max} $  used in   Theorem
\ref{main-thm1}.

 \sk

For every $q\in \rr$, we introduce the exponent
\begin{equation}
\label{defalphaq}
 \alpha(q)={\int_{\zu}    (-\phi)  \, d\mu_{q} \over \int_{\zu}    \log |T'| \,d\mu_{q}}.
\end{equation}
By the Gibbs property of $\mu_{\phi} $ and the ergodicity of $\mu_q
$,
 the measure $\mu_q$ is supported by the level set $\mathcal{E}_{\mup}(\alpha(q))$ (equation \eqref{defemu}). Hence $D_{\mup} (\alpha(q)) \!=\dim  {\mu_q}=\eta_{\phi}(q)+q\alpha(q). $
The  map  $q\mapsto \alpha(q)$ is decreasing, and
\begin{eqnarray}
\nonumber  \lim_{q\to +\infty} \alpha(q)   = \alpha_-,& \ \ \ \lim\limits_{q\to -\infty}\alpha(q)   = \alpha_+,
\\
  \label{alpha1}    \alpha(1)  =\dim {\mu_{\phi}},&       \alpha(0)  = \alpha_{\max}  . 
  \end{eqnarray}
For $\alpha \in ]\alpha_-,\alpha_+[ $, we write    the inverse function
of $q \mapsto \alpha(q)$ as $
\alpha\mapsto q(\alpha).$

\begin{remark}\label{e+}
  By a standard result, we have
  \vspace{-1mm}
  \begin{eqnarray*}
\sup_{y: \  d_{\muphi}(y) \text{ exists}} \ d_{\muphi}(y) = \sup_{v\in \mathcal{M}_{\rm
erg}}\frac{-\int_{\zu} \phi  \, d_{\nu}}{\int _{\zu}   \log
    |T'|d\nu}=\alpha_{+}.
  \end{eqnarray*} 
\end{remark}

 \begin{definition}\label{def-E}
\label{defsets2} Let $s\in \mathbb{R}^+$. We define the sets
\begin{eqnarray}
\label{def3}
     \mathcal{E}_{\ge s} = \{ y \in [0,1]: \,
\underline{d}_{\mu_\phi}(y)\ge  s \}
    \    \mbox{ and } \    \mathcal{E}_{\le s}    = \{y \in [0,1]: \, \underline{d}_{\mu_\phi}(y)\le s \},\\\label{def3bis}
     \mathcal{E}_{> s}   =   \{y \in [0,1]: \, \underline{d}_{\mu_\phi}(y) > s
     \}
 \       \mbox{ and } \    \mathcal{E}_{< s}    = \{y \in [0,1]: \, \underline{d}_{\mu_\phi}(y) < s \}.
\end{eqnarray}
\end{definition}

For sake of simplicity, we omit the dependence on $\phi$ of the quantites $\alpha(q)$, $q(\alpha)$, and   $\mathcal{E}$. Indeed, $\phi$ will be fixed along the proof.

\sk

From    large deviations theory \cite{BMP}, we get the  values for the Hausdorff dimensions of the sets $\mathcal{E} $ in
Definition \ref{def-E}. These
values depend on whether $s$ is located in the increasing or in
the decreasing part of the multifractal spectrum $D_{\mup}$.
\begin{proposition}
\label{propmaj} Let  $\mup$ be a Gibbs measure. Then:
\begin{enumerate}
\item
For every $s < \alpha_{\max}$, $\dim\big( \mathcal{E}_{< s} \big)
= \dim \big(\mathcal{E}_{\le s}  \big) = D_{\mup}(s)$.
\item
For every $s > \alpha_{\max}$, $\dim  \big( \mathcal{E}_{> s} \big)
= \dim \big(\mathcal{E}_{\ge s} \big) = D_{\mup}(s)$.
\end{enumerate}
\end{proposition}

\vspace{-2mm}

 
\section{Results about hitting time}
\label{sec:Hitting}  

\subsection{Orbits and hitting time.}  The proof  of Lemma \ref{lem-orbit}  is left to the reader.
 
\begin{lemma}\label{lem-orbit}
  The following three assertions are equivalent:
  \begin{enumerate}
  \item
  There exists an integer $n_0 \geq 1$ such that $y =T^{n_0}x$ (i.e.
  $y\in \mathcal{O}^+(x)$).
  \item The hitting time
$\tau_r(x,y)$ is bounded for all $r>0$.
\item
There is
a sequence $r_i\to 0$ such that $\tau_{r_i}(x,y)$ is bounded.
\end{enumerate}
\end{lemma}

The next  lemmas investigate the relationship between  
$\Lkap $ and   hitting times.
\begin{lemma}\label{lem-inclusion} For every $\delta>0$, we have the two embedding properties:
\begin{eqnarray}\label{I}\ \ \  \ \
  \mathcal{R}_{< 1/\delta}(x)
\setminus \mathcal{O}^+(x) \subset  & \mathcal{L}^{\delta}(x)  & \subset    \mathcal{R}_{\leq 1/\delta}(x) \sk \\
\label{F}\sk  \mathcal{R}_{> 1/\delta}(x)
 \subset  & \mathcal{F}^{\delta}(x) &  \subset   \mathcal{R}_{ \geq  1/\delta}(x) \cup \mathcal{O}^+(x) .
\end{eqnarray}
\end{lemma}
\begin{proof}
We prove (\ref{I}),  and  (\ref{F}) is deduced  by taking complements.

For the first inclusion, consider $y$ such that   $R(x,y) <  {1}/{\delta}$ and $y \not\in
  \mathcal{O}^+(x)$.  Choose $ \epsi >0 $ such that $R(x,y) < {1}/{\delta}- \epsi$.
  Then by   definition of $R(x,y)$, there is a
  positive real sequence   $(r_i)$ converging to zero such that $
      \tau_{r_i}(x,y) <
       \left( {1}/ {r_i}\right)^{{1}/{\delta}- \epsi}.
$ Consider the  sequence of integers $n_i:=\tau_{r_i}(x,y)$, for all $i\geq 1$. By construction,
  \[r_i < \left(  {n_i}\right)^{-1/({1}/{\delta}- \epsi)}<  {n_i^{-\delta}}.\]
  Thus $T^{n_i}x \in B(y, n_i^{-\delta})$.
  Since $y \not\in
  \mathcal{O}^+(x)$,   Lemma \ref{lem-orbit} yields that  $(n_i)_{i\geq 1} $ is not
  bounded. We deduce that 
  $y\in B(T^{n_i}x,  n_i^{-\delta} )$ for infinitely many increasing integers $n_i$. This proves that   $y\in \mathcal{L}^{\delta}(x)$.

For the second inclusion of \eqref{I},  consider $y\in \mathcal{L}^{\delta}(x)$. By definition, $T^{n_i}x \in B(y,
  n_i^{-\delta})$ for infinitely many integers $(n_i)_{i\geq 1}$. Hence, for  these   $n_i$, we have
  $\tau_{1/n_i^{\delta}}(x,y) \leq n_i$, which  implies that
$$
  R(x,y)\leq \liminf_{i\to \infty} \frac{\log \tau_{1/n_i^{\delta}}(x,y) }{- \log (1/n_i^{\delta})}
   \leq  \liminf_{n_i\to \infty}\frac{\log n_i}{\delta \log n_i} = \frac{1}{\delta}. $$
   This completes the proof.
\end{proof} 

\begin{lemma}\label{lem-noteventually}
 Suppose that  $x$ is not eventually periodic. If $y\in \mathcal{O}^+(x)
 $, then:
 \[
     y\in \mathcal{L}^{\delta}(x)\ \  {\rm if }\ R(y,y)<
  {1}/{\delta}  \qquad {\rm and} \qquad  y\in  \mathcal{F}^{\delta}(x)\  \ {\rm
     if} \      R(y,y)>
 {1}/{\delta}.
 \]
\end{lemma}
Observe that the case where $R(y,y)=1/\delta$ is not determined yet.
\begin{proof}
 Suppose that $y\in\mathcal{O}^+(x)$. Since $x$ is not eventually periodic, there exists a unique positive integer $n_0$
  such that $T^{n_0}x=y $ and $y\not\in \mathcal{O}^+(y)$. The
  rest of the proof is the same as that of Lemma \ref{lem-inclusion}.
\end{proof} 

\begin{lemma}\label{lem-orbit-set}
 Suppose that $\mu \in \mathcal{M}_{inv}$   has no atom. Then for $\mu$-a.e.
 $x$, we have
  \[
    \mathcal{O}^+(x)\subset \mathcal{L}^{\delta}(x)\  {\rm if}\  {1}/{\delta}> \dim  \mu  \quad {\rm and} \quad
     \mathcal{O}^+(x)\subset  \mathcal{F}^{\delta}(x)\  {\rm if}\   {1}/{\delta}<\dim 
     \mu.
 \]
\end{lemma}

\begin{proof}
We remark that  the set of eventually periodic points is a countable
  set,  hence it has a $\mu$-measure equal to zero.   By the 
  Ornstein-Weiss theorem  \cite{OrnsteinWeiss},  
  \begin{equation}
\label{eq0}
  \mbox{for $\mu$-almost all $x$, for every $n\geq 1$, } \ \ \ \        R(T^nx, T^nx) = \dim  \mu.
  \end{equation}
  Hence, for a $\mu$-typical $x$ (which is   not
  eventually periodic),  consider $y\in \mathcal{O}^+(x)$.
  By the same argument as above,  $y=T^{n_0}x $ for some unique integer $ n_0\geq 0$. By \eqref{eq0}, $R(y, y) = \dim 
  \mu$. Applying  Lemma \ref{lem-noteventually},  we find that
  if ${1}/{\delta}>\dim \mu$
  ({\em resp.} ${1}/{\delta}<\dim \mu$) then  $y\in \mathcal{L}^{\delta}(x)$  ({\em resp.}  $y\in \mathcal{F}^{\delta}(x)$).
\end{proof} 

\subsection{Local dimension and hitting time.} 

 Gibbs measures enjoy exponential decay
of correlations   \cite{Ruelle,Parry-Pollicott1990,Liverani-Saussol-Vaienti1998,Baladi2000}.
More precisely, we have the following theorem.

\begin{theorem}
\label{exponentialdecay}
Suppose that  $f:\zu\to\zu $ has bounded variation  and $g:\zu\to\zu$ is
  integrable. Then there exist constants $0<\beta<1 $ and $\Theta>0$, such that for every integer $n\geq 1$,
  \begin{eqnarray*}
    \left| \int fg\circ T^n d\mu_\phi - \int fd\mu_\phi \int g d\mu_\phi
    \right| \leq  \Theta \beta^{n} \left(\int |f| d\mu_\phi + {\rm
    var}(f)\right)\int|g|d\mu_\phi,
  \end{eqnarray*}
  where ${\rm var}(f) $ stands for the total variation of $f$ on $\zu$.
In particular, if $f=\mathbf{1}_A$ and $g=\mathbf{1}_B$ where
$A$ is an interval and $B$ is a measurable set, then for every  $n$,
  \begin{eqnarray}\label{exp-decay}
    \left| \mu_\phi(A\cap T^{-n}B) - \mu_\phi(A)\mu_\phi(B)
    \right| \leq \Theta \beta^{n} \left(\mu_\phi(A) + 2\right)\mu_\phi(B).
  \end{eqnarray}
\end{theorem}

Theorem \ref{exponentialdecay}
 allows us to use the following theorem  which describes the relationship between  hitting time and   local dimension of invariant measures.
\begin{theorem}[\cite{Ga}]\label{Thm-Ga}
  If $(X, T, \mu)$ has superpolynomial decay of correlations and if ${d}_\mu(y)$ exists,
then for $\mu $-almost every $x$ we have
 \[R(x, y) = {d}_\mu(y).\]
\end{theorem}

We return to the study of the Markov map $T$ on the interval $ [0,1]$.

\begin{corollary}\label{cor-fubini}
 Let $\mu_\phi,\mu_\psi$ be two
$T$-invariant Gibbs probability measures on $\zu$ associated with
normalized H\"{o}lder potentials $\phi$ and $\psi$. We have  \begin{eqnarray*}
  \mbox{ for $\mu_{\phi}$-a.e.  $x$, \  for $\mu_\psi$-a.e. $y$, } \
  R(x,y)=d_{\mu_{\phi}}(y)=\frac{-\int_{\zu}  \phi \, d\mu_\psi}{\int_{\zu}  \log
    |T'| \, d\mu_\psi}  .
  \end{eqnarray*}
\end{corollary}
\begin{proof}
 For $\mu_\psi$-almost every $y$,
 $\frac{1}{n}S_n \phi(y) $ tends to $\int_{\zu} \phi\, d\mu_\psi$. Hence,  for $\mu_\psi$-almost every $y$, using the Gibbs property of $\mup $ and $\mu_{\phi_0}$, the H\"older exponent $\tilde{d}_{\mup}(y)$  computed on the basic intervals defined as
 $$\tilde{d}_{\mup}(y):= \lim_{n\to+\infty} \frac{\log \mu(I_n(y))}{\log |I_n(y)|}$$
 exists and is equal  to
$
   \frac{-\int_{\zu}  \phi \,  d\mu_\psi}{\int _{\zu}  \log
    |T'| \, d\mu_\psi}
$. By Theorem 5.1 of \cite{BarralBennasrPeyriere}, for quasi-Bernoulli measures, and in particular for Gibbs measures associated with Markov maps, $\tilde{d}_{\mup}(y)$ coincides with $d_{\mup}(y)$, for $\mu_\psi$-a.e. $y$.
  The lemma then follows by Theorem \ref{Thm-Ga} and the Fubini theorem.
\end{proof} 

\begin{corollary}\label{cor-R}
 Let $\mu_\phi,\mu_\psi$ be two
$T$-invariant Gibbs probability measures on $\zu$ associated with
normalized H\"{o}lder potentials $\phi$ and $\psi$.  Then
  \[
\mbox{for $\mu_\phi\times\mu_\psi-$almost every  $(x,y)$,
 } \ \    R(x, y)= d_{\mu_{\phi}}(y)= \frac{\int _{\zu} (-\phi) {d}\mu_{\psi}}{\int_{\zu}  \log |T'|
   {d}\mu_{\psi}}.
  \]
\end{corollary}

\subsection{First results on covering.}

We introduce the real number
\begin{eqnarray*}\delta(\phi,\psi) :
=\sup \left\{ \delta \geq 0 :  \mu_{\psi}(\mathcal{L}^{\delta}(x))  = 1 \hbox{ for }
\mu_{\phi}\mbox{-almost every } x \right\}.
\end{eqnarray*}

The following proposition, which summarizes the results of the
previous sections,  will be useful when proving the lower bound for
Part I of Theorem \ref{main-thm1}. We can also use it to give a direct proof
for the item 3. of Theorem \ref{main-thm1}. 

\begin{proposition}\label{thm-full-measure} For any normalized H\"{o}lder potentials $\phi$ and $\psi$, we have
\[\delta(\phi,\psi)= \frac{\int_{\zu}   \log |T'| \,  d\mu_{\psi}}{\int_{\zu} \,  (-\phi)
\, d\mu_{\psi}}.\] In particular,  for every $\alpha \in
]\alpha_-,\alpha_+[ $,
 \begin{equation}
 \label{valeurkappa}
 \delta(\phi,\phi_{q(\alpha)})= \sup \left\{ \delta \geq 0:  \mu_{q(\alpha)}(\mathcal{L}^{\delta}(x))  = 1
\ {\rm{ for }}\  \mu_{\phi}-a.e.\ x \right\}  =   {1}/{\alpha}.
 \end{equation}\end{proposition}
\begin{proof}
 Combine  Lemma \ref{lem-inclusion}, Corollary \ref{cor-R} and the definition  \eqref{defalphaq} of $\alpha(q)$.
\end{proof}\medbreak

We are  now able to give a direct proof of the item 3. of Theorem
\ref{main-thm1}.

 \begin{proof} {[Direct proof for Theorem \ref{main-thm1}, 3.]}
 Take the potential
$\psi:=\phi_0= -\log |T'|$. As  already observed, the
corresponding Gibbs measure $\mu_{0}$  is an invariant measure  equivalent to the Lebesgue measure.  Thus, ``$\mu_0$-almost
everywhere" is equivalent to ``Lebesgue-almost everywhere".  Hence, $
\delta(\phi,\psi)$  is also equal to
$$ \sup \left\{ \delta \geq 0:  { Leb}(\mathcal{L}^{\delta}(x))  = 1 \hbox{ for }
\mu_{\phi}-{a.e. } \ x \right\}.$$ From Proposition
\ref{thm-full-measure}, applying \eqref{valeurkappa} with the
measure $\mu_0$, the exponent  $\delta(\phi,\psi)$ coincides with  $
\frac{1}{ \alpha_{\max} }$ (defined by \eqref{defalphamax}). This
concludes the proof.
\end{proof}\medbreak

We  investigate other exponents (recall that $q(\alpha)$ is the inverse function of $\alpha(q)$):
\begin{itemize}

\item
By Proposition \ref{thm-full-measure}, for any $ \epsi>0 $  and $\delta \in
]1/\alpha_+, 1/\alpha_-[ $, for $\mu_{\phi}$-a.e. $x$, we have 
\begin{equation}
\label{crucial}
\mu_{q({1}/{\delta})}(\mathcal{L}^{\delta- \epsi}(x))  = 1.
\end{equation}
 
\item
For  $1/\delta= \alpha(1)= \dim {\mu_{\phi}}$, we have $
q(1/\delta)=1$ and $\mu_q=\mu_1=\mu_{\phi}$. Hence, applying
Proposition \ref{thm-full-measure} and \eqref{alpha1}, we get
\[
\sup \left\{ \delta \geq 0:  \mu_{\phi}(\mathcal{L}^{\delta}(x))  = 1 \hbox{ for }
\mu_{\phi}-a.e.\ x \right\}=\frac{1}{\dim {\mu_{\phi}}}.
\]
 Thus for every   $0< \delta< \frac{1}{\dim {\mu_{\phi}}}$,  we
have $\mu_{\phi}(\mathcal{L}^{\delta}(x))  = 1$ for
$\mu_{\phi}$-a.e. $x$.
\end{itemize}


\section{Multiple-quasi-Bernoulli inequalities and hitting lemmas}
\label{sechitt}

\subsection{Multiple-quasi-Bernoulli inequalities.}  

Let $\phi $ be a normalized potential, i.e. $P(\phi)=0$.
The Gibbs property (\ref{Gibbs}) of $\mu_\phi$ can be
written as
 \begin{equation}\label{GibbsProperty}
\ \ \ \ \    \ \forall  \, x\in [0,1], \ \forall \, n\geq
      1, \qquad
    \frac{1}{\gamma} e^{S_n\phi(x)} \le  \mu_\phi(I_n(x)) \le \gamma
      e^{S_n\phi(x)} .\end{equation} 

It is classical that (\ref{GibbsProperty}) implies  the  
quasi-Bernoulli property of $\mu_\phi$.
\begin{lemma}
For any  basic intervals $A$ and $B$ of   generation $n_A $ and $n_B
$ respectively,
 \begin{equation}\label{QBernoulliProperty}
    \frac{1}{\gamma^3}
     \mu_\phi(A) \mu_\phi(B) \le \mu_\phi(A \cap T^{-n_A} B) \le \gamma^3
     \mu_\phi(A) \mu_\phi(B).
 \end{equation}
\end{lemma}

\begin{proof}
Consider any  $x \in A \cap T^{-n_A} B$. Applying
(\ref{GibbsProperty}) three times, we get   \begin{eqnarray*}
   && \mu_\phi(A \cap T^{-n_A} B)
     \ge    {\gamma^{-1}}e^{S_{n_A}\phi(x) + S_{n_B}(T^{n_A}
    x)}
     \ge  {\gamma^{-3}}   \mu_\phi(A) \mu_\phi(B), \\
        &&  \mu_\phi(A \cap T^{-n_A} B)
     \le  {\gamma} e^{S_{n_A}\phi(x) + S_{n_B}(T^{n_A}
    x)}
     \le  {\gamma^3}   \mu_\phi(A) \mu_\phi(B).  
\end{eqnarray*}
\vskip -4pt
 \end{proof} 

Moreover, by  
(\ref{exp-decay}),  the  following multiple-quasi-Bernoulli
inequalities hold. The same inequalities are  referred to as a {\em
 multi-relation} for the doubling map  in~\cite{FST}.
 
\begin{lemma}\label{th-multi-relation}
  Let $ \mu_\phi $ be the Gibbs measure associated with  a normalized potential
  $\phi $.
   Let $n\in \mathbb{N}, n\geq 1$ and  let $C_0, C_1, \dots, C_k
  $ be  $(k+1)$ basic intervals in $\mathcal{C}_n $.
Then there exist  a constant $M>0$ independent  of the choice of $n$
such that  for all integer $\omega $ large enough ($\beta$ is the constant appearing in \eqref{exp-decay}),
    \begin{equation}
    \label{res1}
  \frac{1}{\gamma^{3}}(1- M \beta^{\omega n})^{k-1} \leq \frac{\mu_\phi \Big(C_0\cap \bigcap\limits_{j=1}^{k}
  T^{-2j\omega n }C_j\Big)}{\prod_{j=0}^{k}
  \mu_\phi (C_j)} \leq  \gamma^{3}(1 + M \beta^{\omega n})^{k-1}.
\end{equation}

\end{lemma}

\begin{proof}  Let $n_j $ be the generation   of $C_j$, and let $\omega$ be an integer large enough that $n_j-2\omega n \leq -1$ for all $0\leq j \leq k$.
Observe  that
$$
     C_0 \cap \bigcap_{j=1}^k  T^{-2j\omega n }  C_j= C_0 \cap T^{-n_0} \mathcal{B}
, \ \ \mbox{where  } 
        \mathcal{B} = \bigcap_{j=1}^k  B_j \ \mbox{ and }B_j:= T^{ n_0- 2j\omega n}
        C_j.$$
 Applying  
(\ref{QBernoulliProperty}) to $A=C_0$ and to each $B=B_j$, we obtain
\begin{equation*}
 {\gamma^{-3}}
     \mu_\phi (C_0) \mu_\phi (B_j) \le \mu_\phi (C_0 \cap T^{-n_0} B_j) \le \gamma^3
     \mu_\phi (C_0) \mu_\phi (B_j).
\end{equation*}
Then, summing over all the $B_j$'s, we get
\begin{equation}\label{MR*}
 {\gamma^{-3}}
     \mu_\phi (C_0) \mu_\phi (\mathcal{B}) \le \mu_\phi (C_0 \cap T^{-n_0} \mathcal{B}) \le \gamma^3
     \mu_\phi (C_0) \mu_\phi (\mathcal{B}).
\end{equation}
The invariance of $\mu_\phi $ implies that
 $$  \mu_\phi (\mathcal{B}) =   \mu_\phi  \Big( \bigcap_{j=1}^{k}  T^{-2j \omega n }
C_j\Big).
$$
Thus, in order to get \eqref{res1},  we need only to prove that for
some constant $M$,
\begin{equation}\label{MR**}
\left(1 -  M \beta^{\omega n}\right)^{k-1}  \le \frac{\mu_\phi
\left( \bigcap_{j=1}^k  T^{-2j \omega n } C_j\right)}{\prod_{j=1}^k
\mu_\phi (C_j)}
       \le  \left(1 +  M \beta^{\omega n}\right)^{k-1}.
 \end{equation}

Recalling the exponential decay of correlation \eqref{exp-decay}, for every choice of two basic intervals $A, B$,
and for  every integer $m$, we have
  \begin{eqnarray*}
    \left| \mu_\phi (A\cap T^{-m}B) - \mu_\phi (A)\mu_\phi (B)
    \right| \leq \Theta \beta^{m} \left(\mu_\phi (A) + 2\right)\mu_\phi (B).
  \end{eqnarray*}
which can be rewritten as
\begin{equation}
\label{intermed1}
   \left( 1 - \Theta \beta^m  \frac{\mu_\phi (A)+2}{\mu_\phi (A)}
   \right)   \le \frac{\mu_\phi (A\cap T^{-m}B)}{\mu_\phi (A)\mu_\phi (B)}
   \le
\left( 1 + \Theta \beta^m  \frac{\mu_\phi (A)+2}{\mu_\phi (A)}
   \right).
\end{equation}
Consider the intervals $C_1, \cdots, C_k$ and observe that
\begin{eqnarray}\label{intersection-decomp}  \bigcap_{j=1}^k  T^{-2j\omega n} C_j =
\big(T^{-2\omega n} C_1 \big) \bigcap  \Big(T^{-2\omega n} \big(
\bigcap_{j=2}^k T^{ -2(j-1)\omega n} C_j  \big) \Big).
\end{eqnarray}
Iterating (\ref{intersection-decomp}), we apply the double-sided
inequality \eqref{intermed1} inductively to obtain
 \begin{eqnarray}\label{multi-prod}
 && \prod_{j=1}^{k-1}
      \left( 1  -   \Theta \beta^{2\omega n} \frac{\mu_\phi (C_j)+2}{\mu_\phi (C_j)}
   \right) \\
   &&  \nonumber \hspace{2mm}
    \le
   \frac{\mu_\phi (\bigcap_{j=1}^k  T^{-2j\omega n} C_j)}{\prod_{j=1}^k \mu_\phi (C_j)}
          \le
     \prod_{j=1}^{k-1}
     \left( 1 + \Theta \beta^{2\omega n} \frac{\mu_\phi (C_j)+2}{\mu_\phi (C_j)}
   \right).
\end{eqnarray}
  By the Gibbs property (\ref{GibbsProperty}), we have for $1\leq j\leq k$,
\[\frac{\mu_\phi (C_j)+2}{\mu_\phi (C_j)}\leq \frac{3}{\mu_\phi (C_j)}\leq 3 \gamma\cdot
e^{-  S_{n_j} \phi(x)}\leq 3 \gamma\cdot
e^{- n_j  ({\rm min}_{x\in\zu}( \phi(x)))} .\]
Since $
\phi$ is normalized, ${\rm min}_{x\in\zu}( \phi(x))$ is negative. Recalling   (\ref{covering1}), we find that
\[\frac{\mu_\phi (C_j)+2}{\mu_\phi (C_j)} \leq 3 \gamma\cdot
e^{- L' n ({\rm min}_{x\in\zu}( \phi(x)))} .\]

For $\omega$ large enough, $e^{- L'({\rm
min}_{x\in\zu}(\phi(x)))}< \beta^{-\omega}$.  Then each  term in the
product on the right side of \eqref{multi-prod} can be bounded from
above by
\[ 1+ \Theta \beta^{2\omega n}  \cdot 3 \gamma
 \beta^{-n\omega}\leq  1+  M \beta^{ \omega n},\]
 where $M$ is some constant depending on $\gamma,\Theta$.
 Then \eqref{MR**} is obtained by multiplying $k-1$ identical terms.
 One gets the lower bound similarily.
\end{proof} 

\subsection{Big hitting probability lemma.}  

Lemma \ref{Big-hitting} illustrates   that   intervals with  small local dimension for $\mu_\phi$ are hit by the balls $B(T^n x, 1/n^\delta)$ with big
probability.
\begin{lemma} \label{Big-hitting}
Let $h$ and $ \epsi$ be two positive real numbers. Consider $N$
distinct basic intervals $C_1, \cdots, C_N$ in $\mathcal{C}_n$
satisfying $\mu_\phi(C_i) \ge |C_i|^{h -  \epsi}$.
Set
$$
 \mathcal{C}_{n,N,h}= \Big\{x \in \zu: \, \exists \, C \in \{C_i\}_{i=1,\ldots, N} \text{ such that }
\tau(x,C) > |C|^{-h}\Big\}.
$$
Then there exists an integer $n_h \in \N$  independent of $N$ such that
\begin{align*}
\mbox{for every   $n\geq n_h$,  } \ \ \mu_{\phi} \big( \mathcal{C}_{n,N,h} \big) \le  2^{-  n}.
\end{align*}
\end{lemma}

%

\begin{proof}
Fix one interval  $\widetilde C$ among  the $N$ basic intervals $C_1, \cdots, C_N$.
Let
$$\mathcal{X}_{\widetilde C}:=\Big\{x\in [0,1] :  \, \forall \ n\leq  |\widetilde C|^{-h},  \ \ \ T^n x \not\in \widetilde C  \Big\}.$$
Obviously we have the embedding property $ \mathcal{C}_{n,N,h} \subset \bigcup_{i=1}^{N} \ \mathcal{X}_{C_i},$
so we are going to bound from above each
$\mup(\mathcal{X}_{\widetilde C})$.  Pick up an integer $\omega$
such that $2\omega >L $ ($L$ appears in 
\eqref{covering}), and set  $m_{\widetilde C}= \lfloor
 |\widetilde C|^{-h}/(2\omega n ) \rfloor$. Then by   definition of
$m_{\widetilde C}$, we have  
$$\mathcal{X}_{\widetilde C}\subset \bigcap_{j=0}^{m_{\widetilde C}} \Big\{x\in [0,1] : T^{2j\omega n} x
\not\in {\widetilde C} \Big\} = \bigcap_{j=0}^{m_{\widetilde C}}   \
\Big(\zu \setminus T^{-2j\omega n}( {\widetilde C}) \Big). $$
We know that
the union of the intervals belonging to $\mathcal{C}_{n}$ is the
whole interval $\zu$. Observe also that the cardinality of
$\mathcal{C}_{n}$ is of order $2^n$.
Let us denote by $\mathcal{C}_n(\widetilde C)$ the subset of $\mathcal{C}_n$ constituted of  the basic intervals disjoint from $\widetilde C$.

Since $2\omega>L $,  the  definition of  $\mathcal{X}_ {\widetilde C}$ implies that
for any point $x\in \mathcal{X}_ {\widetilde C}$, there is a
choice of $m_{\widetilde C}+1$  basic intervals
$(D_0,\dots,D_{m_{\widetilde C}})$ all belonging to $\mathcal{C}_{n}(\widetilde C)$, such that $x\in D_0 \cap T^{-2\omega
n}D_1 \cap \cdots \cap T^{-2m_{\widetilde C}\omega
n}D_{m_{\widetilde C}}$. From this, we deduce that
\begin{align*}
\mu_\phi(\mathcal{X}_ {\widetilde C}) & \le
\sum_{(D_0,\dots,D_{m_{\widetilde C}}) \in (\mathcal{C}_{n}(\widetilde C))^n} \mu_\phi( D_0 \cap T^{-2\omega
n}D_1 \cap \cdots \cap T^{-2m_{\widetilde C}\omega
n}D_{m_{\widetilde C}}).
\end{align*}
 \vspace{-1mm}We choose $\omega$  large enough   that Lemma \ref{th-multi-relation}
can be applied. Inequality \eqref{res1} yields
\begin{eqnarray*}
\mu_\phi(\mathcal{X}_ {\widetilde C})
& \le & \gamma^3( 1 + M\beta^{\omega n})^{m_{\widetilde C}} \sum_{(D_0,\dots,D_{m_{\widetilde C}}) \in (\mathcal{C}_{n}(\widetilde C))^n}  \ \  \prod_{j=0}^{m_{\widetilde C}} \mu_\phi(T^{-2j\omega n}D_i)\\
& = & \gamma^3( 1 + M\beta^{\omega n})^{m_{\widetilde C}} \Big(
\sum_{D \in \mathcal{C}_n(\widetilde C)}  \mu_\phi(D) \Big)
^{m_{\widetilde C}+1}  .\end{eqnarray*}  
Since the intervals of $\mathcal{C}_n$ have
disjoint interiors,  $    \sum_{D \in \mathcal{C}_n(\widetilde C)}  \mu_\phi(D) \leq 1 - \mu_\phi(\widetilde C).$
Hence \vspace{-1mm}
\begin{eqnarray*}
\mu_\phi(\mathcal{X}_ {\widetilde C})
& \leq & \gamma^3 (1 + M\beta^{\omega n})^{m_{\widetilde C}}    (1 - \mu_\phi(\widetilde C) )^{m_{\widetilde C}+1}\\
& \le    & \frac{\gamma^3 }{1+M\beta^{\omega n}}  \left((1 +
M\beta^{\omega n})(1 - \mu_\phi(\widetilde C))\right)^{m_{\widetilde
C}+1}.
\end{eqnarray*}
 If $\omega$ is chosen large enough,
 \begin{equation}
\label{attention}
 (1 + M\beta^{\omega n})(1 - \mu_\phi(\widetilde C)) \leq 1 -   \mu_\phi(\widetilde C)/2.
\end{equation}
Thus we finally obtain
 \begin{equation}
\label{attention2}
\mu_\phi(\mathcal{X}_ {\widetilde C}) \leq   \frac{\gamma^3 }{1+M\beta^{\omega n}}
\big ( 1 -   {\mu_\phi(\widetilde C)} /2 \big )^{m_{\widetilde C}+1} .
\end{equation}
Here, we emphasize that $\omega$ can be chosen  large enough  that
\eqref{attention}, and thus  \eqref{attention2}, can be realized
simultaneously for all $\widetilde C$ and for all $n$. In fact, from
the Gibbs property (\ref{Gibbs}) of  $\mu_\phi$, there
exists a maximal exponent $H>0$ such that 
$$\mbox{for every basic interval
$\widetilde C$ of any generation $n$, } \ \mu_\phi(\widetilde C) \geq |\widetilde C| ^H \geq  L^{-H}2^{-n H}.$$ Thus, for all $\widetilde C \in \mathcal{C}_n$,
$\displaystyle
  \frac{1 - \frac{1}{2} {\mu_\phi(\widetilde C)}  }{1 - \mu_\phi(\widetilde
  C)} \geq \frac{1-\frac{1}{2}L^{-H}2^{-n H}}{1-L^{-H}2^{-n H}}$. 
So,  we can choose     $\omega$ so that
$$
1+M\beta^{\omega n} \leq \frac{1-\frac{1}{2}L^{-H}2^{-n
H}}{1-L^{-H}2^{-n H}} \leq  \frac{1 - \frac{1}{2} {\mu_\phi(\widetilde C)}  }{1 - \mu_\phi(\widetilde
  C)},
$$
which implies \eqref{attention}   for all $\widetilde C$. Summing over all $\widetilde C \in \{C_1,\ldots, C_N\}$,  by
\eqref{attention2} and the definition of $m_{\widetilde C}$, we have
\begin{align*}
 \mu_{\phi}\Big(  \mathcal{C}_{n,N,h} \Big)
&  \le\frac{\gamma^3 }{1+M\beta^{\omega n}}  \sum_{\widetilde C}
\big ( 1 -  {\mu_\phi(\widetilde C)} /2 \big )^{{| \widetilde C|^{-h}}/{(2\omega n)}}\\
&\le \frac{\gamma^3 }{1+M\beta^{\omega n}} \sum_{\widetilde C}
\big ( 1 -  {\mu_\phi(\widetilde C)} /2 \big )^{{|\widetilde C|^{- \epsi}}/(2\omega n\mu_\phi(\widetilde C))}\\
&=\frac{\gamma^3 }{1+M\beta^{\omega n}}   \sum_{\widetilde C} \exp
\Big ( \frac{|\widetilde C|^{- \epsi}}{2\omega n\mu_\phi(\widetilde
C)}
\log \big ( 1 -  {\mu_\phi(\widetilde C)} /2 \big ) \Big ) \\
& \le \frac{\gamma^3 }{1+M\beta^{\omega n}}  \sum_{\widetilde C}
\exp \Big ( \frac{ -|\widetilde C|^{- \epsi}}{4\omega n}\Big ).
\end{align*}
 Now, recalling   \eqref{covering}, we have $ |\widetilde C|^{- \epsi} \geq L^{-\epsi}2^{\epsilon n} $. Since the number $N$ of possible
 choices for $\widetilde C$ is less than $L\cdot 2^n$,  we have
\begin{align*}
 \mu_{\phi}\Big (
 \mathcal{C}_{n,N,h}\Big)\le \frac{\gamma^3 }{1+M\beta^{\omega n}}  \sum_{\widetilde C} \exp \left(\frac{ - 2^{\epsilon n}}{4\omega n L^{\epsi}}\right) \le   \frac{L\gamma^3}{1+M\beta^{\omega n}}  2^{ \frac{1}{\log 2} \big(n \log 2- \frac{2^{\epsilon n}}{4\omega n L^{\epsi}} \big)  } .
\end{align*}
This last term is independent of $N$ and less than $2^{-  n}$ for sufficiently large $n$.
\end{proof} 

\subsection{Small hitting probability lemma.} We now study the
probability of hitting points with high local dimension for
$\mu_\phi$.   The arguments are close to those of
\cite{FST}.

\begin{lemma} \label{Small hitting}
Let $0<a<1$, $0< c <b<1$ and $\eta >b-c$. Consider
$2^{b  n}$ different basic intervals  $C_1, \cdots, C_{2^{b n}}$ in
$\C_n $.
 Assume that for every $j\in \{1,\cdots, 2^{b n}\}$,
$$ \mu_\phi(C_j) \le 2^{-(a+\eta)n}.$$
Set 
$$\displaystyle \mathcal{X}_{a,b,c}:=\Big\{x: \, \tau(x,C_i) \le  2^{a n} \text{ for } 2^{c n}  \text{
distinct intervals among the}\ \{C_i\}_{i=1,.,2^{b n}}\Big\}.$$
Then there exists  an integer $n_{a,b,c}  \in \N$ such that as soon as $n\geq n_{{a,b,c}}$,
$$ \mu_{\phi} ( \mathcal{X}_{a,b,c}) \le 2^{ -n}.$$
\end{lemma}

\begin{proof} Let us denote   $K := 2^{a n}, P := 2^{b n} , N := 2^{c n} $.
When $x\in \mathcal{X}_{a,b,c}$,  there exist  $N$ integers $0< \ell_1<\ell_2<\cdots
<\ell_N \leq  K$ and $N$ different basic intervals $C_{i_1}, C_{i_2}, \cdots,
C_{i_N}$ such that
\begin{equation}
\label{deftau0}
    T^{\ell_1}x \in C_{i_1}, \ \
    T^{\ell_2}x \in C_{i_2},\ \ \cdots, \ \  T^{\ell_N}x \in
    C_{i_N}.
\end{equation}
Let $N':=\lfloor N/(2\omega n) \rfloor $ and let $(t_p)_{p=1}^{N'} $
be a subset of $(\ell_j)_{j\in\{1,\cdots, N\}} $ defined by
$ \ t_{p } = \ell_{2 \omega n p}.$
Let $j_{p }$ be the unique index $i$ such that $T^{t_{p }}x
\in C_{i}$ in \eqref{deftau0}.
If $x\in \mathcal{X}_{a,b,c}$,  
\begin{equation}\label{Eq-SmallProbability}
    T^{t_1}x \in C_{j_1}, \ \
    T^{t_2}x \in C_{j_2},\ \ \cdots, \ \  T^{t_{N'}}x \in
    C_{j_{N'}},
\end{equation}
where $C_{j_1},\dots, C_{j_{N'}} $ are $N'$ different basic
intervals  among the intervals  $C_1, \ldots, C_P$.

\sk

Fix now $N'$  basic intervals  $C_{j_1},\dots, C_{j_{N'}} $  among
the intervals  $C_1, \ldots, C_P$ and  fix also the integers $ t_1 <
\ldots < t_{N'} \leq K$. Consider the set $\widetilde{
\mathcal{X} }$ of points $x$ such that
\eqref{Eq-SmallProbability} is satisfied.  This set
$\widetilde{\mathcal{X}}$ depends on $a$, $b$, $c$, and on the
intervals and the integers we have chosen. As said above,
$\mathcal{X}_{a,b,c} \subset \bigcup\widetilde{\mathcal{X} }$, where
the union is taken over all possible choices of parameters
$C_{j_1},\dots, C_{j_{N'}} $   and $ t_1 < \ldots < t_{N'}$. In
order to bound from above the $\mup$-measure of $\mathcal{X}_{a,b,c}
$, we will first study the $\mup$-measure of one set
$\widetilde{\mathcal{X}}$.
 Applying (\ref{MR**}) again and using the same arguments as in
 Lemma \ref{Big-hitting}, we see that the $\mu_\phi$-measure of
 $ \widetilde{\mathcal{X} }$ 
 is bounded from above by
\begin{equation}
\label{bigfin} \mup(\widetilde{\mathcal{X}}) \leq  \max_{1\le i\le
L} \mu_\phi(C_i)^{N'} (1+M\beta^{2\omega n})^{N'}.
\end{equation}

It remains us to estimate the maximal number of choices for the
associated intervals  $C_{j_1},\dots, C_{j_{N'}} $ and integers
$(t_1,\ldots, t_{N'})$.
We have ${P \choose N'}$ possible choices for the $N'$
different basic intervals  among the list of  $P$ intervals $C_1,
\ldots, C_P$, and there are at most ${K \choose N'}$ choices for the
integers  $t_1<t_2<\cdots <t_{N'}<K$. Finally there are $N'!$ ways
to arrange the $N'$ intervals.
Combining this  and \eqref{bigfin}, we find 
$$
\mu_\phi(\mathcal{X}_{a,b,c}) \, \leq  \,
\sum_{\widetilde{\mathcal{X}}} \mup(\widetilde{\mathcal{X}})  \,
\leq \,  {P \choose N'} {K \choose N'} \cdot N'!
\cdot\max_{C_i}\mu_\phi(C_i)^{N'}\cdot (1+M\beta^{\omega n})^{N'}.
$$ Since
\vspace{-2mm}
$$ {P\choose N'} {K \choose N'} \cdot N'! =\frac{P!}{(P-N')!}
\cdot\frac{K!}{(K-N')! } \cdot \frac{1}{N'!},
$$
 using the estimates
$\frac{P!}{(P-N')!} \le P^{N'}$, $  \frac{K!}{(K-N')!} \le  K^{N'}$, $ \frac{1}{N'!} \le \xi
\cdot \frac{e^{N'}}{{N'}^{N'}}$
 for some universal constant $\xi$, we  conclude that
\vspace{-1mm}\begin{eqnarray*}
\mu_\phi(\mathcal{X}_{a,b,c})  \leq  \xi  \cdot P^{N'}\cdot
K^{N'}\cdot e^{N'}\cdot N'^{-N'}\cdot  (\max_{C_i}
\mu_\phi(C_i))^{N'}\cdot (1+M\beta^{\omega n})^{N'}  .
\end{eqnarray*}
\vspace{-1mm}Replacing all constants $K$, $P$, $N$ by their values, we get
\begin{eqnarray*}
\mu_\phi(\mathcal{X}_{a,b,c})  \leq  \xi    \cdot
\left(2^{bn}\cdot 2^{an}\cdot e \cdot (N' )^{-1}\cdot
2^{-(a+\eta)n}\cdot (1+M\beta^{\omega n})\right)^{N'}.
\end{eqnarray*}
By definition of $N'$, we have  $(N')^{-1} \leq \frac{2\omega n}{N}
= 2\omega n 2^{-cn}$ when $\omega$ is large enough. Consequently,
the last inequality yields
 \begin{eqnarray*}
\mu_\phi(\mathcal{X}_{a,b,c}) 
  \le   \xi    \cdot  \left ( {e\cdot 2\omega n  \cdot (1+M\beta^{\omega n})  \cdot 2^{(b-c -
\eta) n}} \right )^{N'}.
\end{eqnarray*}
By assumption, $\eta>b-c$,  so the quantity between brackets tends to
0 exponentially fast. In particular, it is less than 1/2.
Using the fact that  $  N' \geq   {2^{cn}}/{2\omega n}$, we get  
\begin{eqnarray*}
\mu_\phi(\mathcal{X}_{a,b,c}) \leq  \xi   \cdot   2^{
-\frac{2^{cn}}{2\omega n}}.
\end{eqnarray*}
The right term in the above inequality is less than   $2^{-  n}$ when $n$ becomes large.
\end{proof} 


\vspace{-1mm}

\section{Part I of the spectrum: $1/\delta < \alpha(1)=\dim  \mu_\phi$}
\label{sec-lower1}

\subsection{Upper bound for  $ \dim  \mathcal{L}^\delta(x) $.}  
By (\ref{I}), we need only to show the following.
\begin{proposition}\label{upper1}
For every $0<s\le \dim {\mu_\phi}$, for every $x\in \zu$,  we have
\begin{equation}
\label{UpperBoundIhn2} \dim    \big(  \mathcal{R}_{\leq s }(x)  \big) \le s.
  \end{equation}
\end{proposition}

\begin{proof}
Notice that in the definition of $ R(x,y)$, one can
replace the limit process of $r\to 0$ by the sequence $2^{-n}$ with
$n\to \infty $. Then for $x\in\zu$ and  any  $a>s$,
$$
 \mathcal{R}_{\leq s }(x)  \subset \limsup_{n\to \infty}
     \left\{y\,:\, \tau_{2^{-n}}(x, y) \le 2^{an}\right\}   .
$$
In other words, given $y\in  \mathcal{R}_{\leq s }(x)  $, there exists an  integer $1\leq k_n\leq 2^{an}$ such that  $y \in B(T^{k_n} x,2^{-n})$, for infinitely many $n$.  Assume that the sequence of integers $(k_n)$ tends to infinity. Using that  $2^{-n} \leq (k_n)^{-1/a}$ for such a couple of integers $(k_n,n)$, we have $y \in B(T^{k_n} x,(k_n)^{-1/a})$  for
infinitely many integers $k_n$. 
Hence
$$
 \mathcal{R}_{\leq s }(x)
     \subset \limsup_{n\to \infty}  B(T^k
     x, k^{-1/a}) .
\vspace{-1mm}
$$

For each integer $n$, we deduce that    the set of balls $\{B(T^k x, k^{-1/a})\}_{k\geq n}$ forms a
covering of $ \mathcal{R}_{\leq s }(x)  $ by   intervals of length
smaller than $n^{-1/a}$.  Let $\mathcal{H}^a_\epsi$ stand for the
$a$-Hausdorff measure obtained by using coverings by balls of
size less than $\epsi$. Using $\{B(T^k x,
k^{-1/a})\}_{k\geq n}$  as covering,  we see that for any $a'>a$,
$$\mathcal{H}^{a'}_{n^{-1/a}}  ( \mathcal{R}_{\leq s }(x)  )
\leq  \sum_{k\geq n}   |B(T^k x, k^{-1/a})|^{a'}  \leq 2^{a'/a}
\sum_{k\geq n}     k^{-a'/a}    \leq \xi' \, n ^{1-a'/a},$$ which
tends to 0 when $n$ tends to infinity. Here $\xi'$ is a universal
constant. We deduce that the $a'$-Hausdorff  measure of
$\mathcal{R}_{\leq s }(x)$ is necessarily 0. Thus $\dim 
\mathcal{R}_{\leq s }(x) \leq a'$.  Since this holds for any $a'>a$,
and then for any $a>s$,   \eqref{UpperBoundIhn2} follows.
\end{proof}

%


\vspace{-1mm}

\subsection{Lower bound for $ \dim  \mathcal{L}^\delta(x) $.} 

Let   $(x_n)_{n\geq 1}$ be a sequence in $\zu$, and let  $(l_n)_{n\geq 1}$  be a positive decreasing sequence.
Consider the limsup sets of the form
$$\mathcal{L}_\zeta := \bigcap_{N\geq 1} \ \bigcup_{n\geq N} \ B(x_n,(l_n)^\zeta).$$
Provided that   $\mu_\phi(\mathcal{L}_{\zeta_0}) =1$  for some $\zeta_0>0$, the dimensions of $\mathcal{L}_\zeta $
can be bounded from below using the  heterogeneous
ubiquity  theorems developed in  \cite{BS}.  To apply such theorems,
some assumptions need to be checked for $\mu_\phi$. We refer   to
Definition 2 of \cite{BS} for the precise description of these
assumptions. We  explain now why these
assumptions are fulfilled in our framework.
From Theorem  1.11(2) of Baladi \cite{Baladi2000}, Theorem 7.1 of Philipp and Stout \cite{PHILIPPSTOUT},
we deduce the following properties for $\mu_\phi$.

\begin{theorem}
\label{thhypotheses} Assume that the   potential $\phi$ associated
with $\mu_\phi$ is H\"olderian.
 There exists a non-decreasing continuous function $\chi$
defined on $\mathbb{R}_+$ with the   properties:
\begin{itemize}
\item
 $\chi (0) =0$, $r\mapsto r^{-\chi(r)}$ is non-increasing near $0^+$,
\item
 $\lim_{r\to 0^+}
r^{-\chi(r)}=+\infty$, and $\forall \, \epsi>0$, $r\mapsto
r^{\epsi-\chi (r)}$ is non-decreasing near 0,
\end{itemize}
such that  for $\mu_\phi$-almost every $y \in [0,1]^d$, there
exists $ r(y)>0,$ such that 
\begin{eqnarray}
\label{scaling1}
\mbox{for all $0<r\leq r(y)$, } \ \ \  r^{\dim \mu_\phi+\chi(r)}\le  & \mu_\phi \big
(B(y,r) \big ) & \le r^{\dim \mu_\phi-\chi(r)}.
\end{eqnarray}
\end{theorem}

 Property \eqref{scaling1} shall be viewed as  an  illustration of the iterated logarithm law for invariant measures. By the theorems of  \cite{Baladi2000}  and  \cite{PHILIPPSTOUT}, one can take $\chi$
 equal to
\begin{equation}
\label{defchi} \chi(0)=0  \  \mbox{ and }  \ \ \chi: r\mapsto   
\Big({\frac{\log\log|\log (r)|}{|\log r|}}\Big)^{1/2} \ \mbox { if }r>0.
\end{equation}

 In the previous section, we proved the following:   for any $\delta$ such that $ 1/\delta> \dim \mu_\phi =\alpha(1)$,   for $\mu_\phi$-almost every $x\in \zu$,  $
\mu_{\phi}(\mathcal{L}^{\delta}(x))  = 1.$
Theorem \ref{thhypotheses} and the quasi-Bernoulli property of   $\mu_\phi$ and $\mu_q$ imply that the conditions of Definition 2 of \cite{BS} are fulfilled for $\mu_\phi$-a.e.  $x\in \zu$. We can then apply   the heterogeneous ubiquity Theorem  4  of  \cite{BS} to get the following lower bound.

\begin{theorem}
\label{ubiq} For any $\delta$ such that $ 1/\delta> \dim  \,
\mu_\phi  $, for $\mu_{\phi}$-a.e. $ x\in \zu$, we
have
 \begin{equation}
 \label{minor1}
 \mbox{for every $\zeta>1$, } \ \ \  \dim (\mathcal{L}^{\zeta \cdot \delta }(x))
\geq  ({\dim {\mu_{\phi}}})/{\zeta}.
\end{equation}
\end{theorem}
 
By considering an increasing  countable sequence $(\delta_n)$ tending to $\delta_0= 1/\dim  \mu_\phi$ and applying Theorem \ref{ubiq} to each $\delta_n$, we get immediately:
\begin{corollary}
For $\mu_{\phi}$-almost every $ x$,  for every $\zeta >1 $,
$$
 \dim (\mathcal{L}^{\zeta \cdot \delta_0}(x)) \geq
 ({\dim {\mu_{\phi}}})/{\zeta}=  {1}/({ \zeta \cdot \delta_0 }).
$$
In other words, for every $\delta$ such that $1/\delta < \dim  \mu_\phi$, we have the lower bound
\[
 \dim (\mathcal{L}^{\delta}(x)) \geq
 {1}/{\delta}.
\]
\end{corollary}

\vspace{-1mm}

\section{Part IV of the spectrum: $1/\delta > \alpha_+$}
\label{secpart4}

\begin{proposition}\label{III-u}
 Let $s \ge 0$. For $\mu_\phi$-almost every  $x$,
\begin{equation}\label{EarlyHitSet1}  \mathcal{R}_{\ge s}(x)  \subset    \mathcal{E}_{\ge s}  . \end{equation}
 Moreover, for any Gibbs measure $\mu_{\psi}$ on $[0,1]$, 
$$\mbox{  for
$\mu_\phi$-almost every  $x\in \zu$,   }  \ \  \mathcal{R}_{\ge s}(x)   \stackrel{\mu_{\psi}}{ =}    \mathcal{E}_{\ge s}    ,$$
where the  equality means that the two sets differ from a set of $\mu_{\psi}$-measure zero.
\end{proposition}

\begin{remark}\label{remark-depend}
  The full $\mu_\phi$-measure set  satisfying the first assertion of
  Proposition \ref{III-u} depends on $s$. 
\end{remark}

\begin{proof} The case $s=0$ is obvious,  we assume that $s>0$.
For any integer $n\geq 1$, let $I_n(y)$ be the basic interval in
$\C_n $ containing $y$. Observe that a priori the generation of $I_n(y)$ is  {\em not} n. For any real number $\epsi>0$, we introduce
the sets
 \begin{eqnarray*}
  \mathcal{R}_{n, s,\e}(x) =  \{y: \tau(x, I_n(y)) \ge |I_n(y)|^{s-\e}\}  
  , \ \   \mathcal{E}_{n, s,\e}  = \{y: \mu_\phi(I_n(y)) \le |I_n(y)|^{s-2\e} \}.
     \end{eqnarray*}
 By definition of $ R(x,y) $ and $
\underline{d}_{\mu_\phi}  (y) $, 
we have
$$
   \mathcal{R}_{\ge s}(x) = \bigcap_{\e>0} \liminf_{n \to \infty}  \mathcal{R}_{n, s,\e}(x)  \  \ \mbox{ and } \
 \
   \mathcal{E}_{\ge s}   = \bigcap_{\e>0} \liminf_{n \to \infty}  \mathcal{E}_{n, s,\e}  .
$$
In order to prove \eqref{EarlyHitSet1}, it is sufficient to prove that for $\mu_\phi$-almost every $x$, there
exists some integer $n(x)$ such that
\begin{equation}
\label{inclusion}
    \forall \ n\geq n(x), \ \ \ \ \     \mathcal{R}_{n,s,\e}(x) \subset   \mathcal{E}_{n,
    s,\e}.
\end{equation}
Notice that $ \mathcal{E}_{n,s,\e} ^c$ is the union of basic
intervals $C$ in $\mathcal{C}_n$ such that $\mu_\phi(C)
>|C|^{s-2\e}$. Let $\mathcal{D}_{n,s, \e}:=\{C_1, \cdots, C_N\}$
be the set of  these basic intervals. Using Lemma
\ref{Big-hitting} to the   basic intervals
$\mathcal{D}_{n, s,\e}$ and to $h=  s - \epsi $, we  see that for $n$   larger than some  $n_{s,\epsi}$,
$$
   P_n:=    \mu_\phi \Big( \Big\{ x: \exists \,  C \in \mathcal{D}_{n,s, \e} \ \mbox{\rm such that }\
     \tau(x, C) \ge  |C|^{s-\e}\Big\} \Big) \leq 2^{-   n}.
$$

The
sum over $n\geq n_{s,\epsi}$ of the $P_n$'s is  finite.
Applying the
Borel-Cantelli lemma,   the following holds for $\mu_\phi$-a.e.\ $x$:   there
exists an integer $n(x)$ such that 
$$\forall \, n\geq n(x),   \ \ \forall \, C \in \mathcal{D}_{n, s,\e}, \ \ \tau(x, C)<
|C|^{s-\e}.$$
Hence, if   $n\geq n(x)$, then    every $C \in \mathcal{D}_{n, s,\e}$ is included in $
\mathcal{R}_{n,s,\e}(x)^c$.  This yields that  $
\mathcal{E}_{n,s,\e}^c \subset  \mathcal{R}_{n,s,\e}(x)^c$, which is
clearly equivalent to \eqref{inclusion}. Then the first assertion
\eqref{EarlyHitSet1} of Proposition \ref{III-u} follows.

\mk

To prove the second assertion, using the ergodicity of $\mu_{\psi}$,  it suffices to show that
$$
\mbox{ for
$\mu_\phi$-almost every  $x$,  } \ \ \mu_{\psi}\Big(\Big\{y\in [0,1]:    \underline{d}_{\mu_\phi}(y)\ge s \ \mbox{ and } \    R(x, y)<s
\Big\}\Big)=0.
$$
This last  statement  is directly deduced from  Corollary \ref{cor-fubini}.
\end{proof}\medbreak

We are now ready to prove some of the statements of Theorem \ref{main-thm1}.

\begin{proof}   $[$Part IV of the spectrum: Theorem  \ref{main-thm1}, 4.$]$
 By
Remark \ref{e+} and Proposition \ref{III-u}, for each $s> \alpha_+ $
for $\mu_\phi$-almost every $x\in\zu$, we have
\[
   \mathcal{R}_{\geq s}(x)= \Big\{y \in [0,1]: \, R(x,y) \geq s \Big\} = \emptyset,
\]
i.e. there is no point with hitting times larger than $s>\alpha_+$.
Then, applying  formula (\ref{F}) and Lemma \ref{lem-orbit-set}, we
deduce that when  $1/\delta > \alpha_{+}$, for
$\mu_\phi$-almost every $x\in\zu$, $\mathcal{F}^\delta(x)=\emptyset$
 and thus $\Lkap=\zu$. But as mentioned in Remark
\ref{remark-depend}, the full $\mu_\phi$-measure set depends on
$\delta$. To solve this problem, i.e. to get $\mathcal{F}^\delta(x)
= \emptyset$ for every $\delta$ satisfying $1/\delta >\alpha_+$, we
take a sequence $(\delta_n)_{n\geq 1}$ such that $(1/\delta_n) $ is
dense in $]\alpha_+, \infty[ $. By taking intersection of countable
full $\mup$-measure sets, we obtain that for $\mu_\phi$-almost every
$x\in\zu$, for all $n$, $\mathcal{F}^{\delta_n}(x)=\emptyset$ and
$\mathcal{L}^{\delta_n}(x)=\zu$.  Finally, the case of an  arbitrary
$\delta$  such that $1/\delta >\alpha_+$ is obtained by using the
monotonicity of the  sets $ \mathcal{F}^{\delta }(x)$ and
$\mathcal{L}^{\delta}(x) $ with respect to $\delta$.
\end{proof} 

\section{  Part III of the spectrum: $\alpha_{\max} < 1/\delta \leq \alpha_+$}
\label{secpart3}

In this  short section, we gather the previous results  to obtain
Part III of the spectrum and item 3. of Theorem   \ref{main-thm1}.
   We adopt the notations of   Section \ref{sechitt}. Let $\delta $  be such that  $\alpha_{\max} < 1/\delta  \leq \alpha_+$,  and consider the unique real number $ q(1/\delta)$. Then the associated   Gibbs measure $\mu_{q(1/\delta)} $ is supported on the level set  
$ \mathcal{E}_{\mu_\phi}(1/\delta) ,$
which has Hausdorff dimension $D_{\mu_\phi}  (1/\delta)$.

Further, we    apply  the second part of  Proposition \ref{III-u} to the
  measure $\mu_{\psi}=\mu_{q(1/\delta)}$. Then for  $\mu_{\phi} $-almost every $x$,  the measure
$\mu_{q(1/\delta)}$ is also supported on the set $ \mathcal{R}_{\geq
1/\delta}(x) .$ In particular, we conclude that  $ \dim  \,  \mathcal{R}_{\geq 1/\delta }(x)  \geq   \dim  \, \mu_{q(1/\delta )} $.

Now,  consider a  countable sequence   $(\delta_n)_{n\geq 1}$ such
that $1/\delta_n$ is dense in the interval $[\alpha_{\max} ,
\alpha_+]$. The above argument applies to each $\delta_n$. Taking a countable  intersection of   full $\mup$-measure
sets, we find a set of full  $\mup$-measure of points $x$ such that for all $n\geq 1$, $\mu_{q(1/\delta_n )}$ is
also supported on the set $ \mathcal{R}_{\geq 1/\delta_n}(x) .$

\sk

Let us fix  $\delta_0 $ such that  $\alpha_{\max} < 1/\delta_0  \leq \alpha_+$, and consider a subsequence  $(\delta_{\varphi(n)})_{n\geq 1}$  decreasing to $\delta_0$. By  (\ref{F}),  for every integer $n$,
\[
   \dim ( \mathcal{F}^{\delta_0}) \geq \dim  \mathcal{R}_{\geq 1/\delta_{\varphi(n)}}(x)  \geq   \dim   (\mu_{q(1/\delta_{\varphi(n)} )})  = D_{\mu_\phi}( {1}/{\delta_{\varphi(n)}}).
\]
Using the continuity of $D_{\mu_\phi}$ on its support, we see that for $\mu_\phi$-almost every $x\in \zu$,
\[
   \dim ( \mathcal{F}^{\delta_0}) \geq \dim (\mu_{q(1/\delta_0 )}) = D_{\mu_\phi}( {1}/{\delta_0}).
\]

Conversely,  by choosing an increasing subsequence $(\delta_{\varphi(n)})_{n\geq 1}$ converging to $\delta_0$, by (\ref{F}) and Proposition \ref{III-u}, we
have for $\mu_{\phi} $-almost every $x$,
\[
   \dim ( \mathcal{F}^{\delta_0}) \leq \inf_{n}  \dim \big( \mathcal{E}_{\geq {1}/{\delta_{\varphi(n)}}} \big) = \inf_{n}   D_{\mu_\phi}({1}/{\delta_{\varphi(n)}}) = D_{\mu_\phi}({1}/{\delta_0}).
\]

  This completes the proofs for the Part III and for  item 3. of Theorem \ref{main-thm1}.

 
\section{Part II  of the spectrum:  $\dim \mu_\phi< 1/\delta \leq \alpha_{\max}$}
\label{secpart2}


\begin{proposition}
\label{propfinal1}
 If $\dim {\mu_\phi}<s<\alpha_{\max}$ then for $\mu_\phi$-almost every $x$ we have
\begin{equation}\label{LowerBoundIhn1} \dim   \mathcal{R}_{  \le  s} (x) \ge D_{\mu_\phi}(s).
  \end{equation}
\end{proposition}

\begin{proof}
For $\dim {\mu_\phi}<s<\alpha_{\max}$, there exists a real number
${q_s}>0$ such that 
$${\int (-\phi)  \, d\mu_{{q_s}} \over \int \log |T'| \, d\mu_{{q_s}}}=s.
$$
By the
Gibbs property of $\mu_{\phi} $ and the ergodicity of $\mu_{q_s} $,
 the measure $\mu_{q_s}$ is supported on  $\mathcal{E}_{\mup}(s)$. Then by Corollary \ref{cor-R} applied to $\mup$ and $\mu_{q_s}$, for $\mu_\phi$-a.e. $x$ we have
\begin{equation*}\dim   \, \mathcal{R}_{  \le  s} (x)  \ge \dim {\mu_{q_s}}=D_{\mu_\phi}(s).
  \end{equation*}
\vskip -8pt
\end{proof} 

We finish by bounding from above the spectrum $ \dim 
\mathcal{L}^\delta(x) $.

\begin{proposition}
\label{propfinal2}
If $\dim {\mu_\phi}<s<\alpha_{max}$ then for $\mu_\phi$-almost every $x$ we
have
\begin{equation}
\label{UpperBoundIhn1} \dim  \mathcal{R}_{\leq s}(x)  \le D_{\mu_\phi}(s).
  \end{equation}
\end{proposition}

\begin{proof}
Fix $s\in (\dim  \mup, \alpha_{\max})$, and let us decompose $\mathcal{R}_{\leq s}(x)$ into
$$
    \mathcal{R}_{\leq s}(x) =  \left( \mathcal{R}_{\leq s}(x)  \cap \mathcal{E}_{\leq s}
     \right) \
      \bigcup  \ \left(\mathcal{R}_{\leq s}(x) \cap  \mathcal{E}_{>s}  \right).
$$
Since $s$ lies in the increasing part of the spectrum, by
Proposition \ref{propmaj},  we have the upper bound $\dim 
\mathcal{E}_{\leq s}  \le D_{\mu_\phi}(s)$. Thus,
to obtain \eqref{UpperBoundIhn1}, it suffices to prove that
$$
\dim  \left( \mathcal{R}_{\leq s}(x) \cap \mathcal{E}_{>s} \right) \le D_{\mu_\phi}(s).
$$

Recall that  $\C_n $ forms a covering of $\zu$ by basic intervals of size $\sim 2^{-n}$, these intervals having disjoint interiors.

Let $0< h'<h'' $. We define the subsets   $\mathcal{C}_n(h',h'')$   of  $\mathcal{C}_n$ and $\mathcal{Y}_n(h',h'')$ of $\zu$
\begin{eqnarray*}
\mathcal{C}_n(h',h'')  & =  & \{ C\in \mathcal{C}_n:  |C|^{h''} \le {\mu_\phi}(C) \le
 |C|^{h'} \}   \ \  \subset \mathcal{C}_n,\\ 
\mathcal{Y}_n(h',h'') &  = &  \{ y\in \zu: \exists \, C\in \mathcal{C}_n (h',h'') \mbox{ such that } y\in C \}  \ \ \subset\zu.
\end{eqnarray*}

\begin{lemma}
For every $\epsi>0$, there exists an integer $n_{h',h'',\e}$ large enough so that as soon as $n\geq n_{h',h'',\e}$,
\begin{eqnarray}\label{cardinality1}
  {\rm Card} \, \mathcal{C}_n
(h', h'') &  \le &  2^{n  (D_{\mu_\phi}(h'')+\epsi) }
   \quad {\rm if } \ \ \  h''< \alpha_{\max},  \\
   \label{cardinality2}
   \ \ \ \  {\rm Card} \, \mathcal{C}_n
(h', h'')  & \le  & 2^{n  (D_{\mu_\phi}(h') +\epsi)}
   \,  \quad {\rm if } \ \ \  h' > \alpha_{\max}.
\end{eqnarray}
\end{lemma}
These properties  follow
again  from  standard large deviations properties (see \cite{BMP}).

\mk

%

%



Let $\zeta>0$ be a positive real number that we will soon choose in a suitable manner.
Set $h'_1=s $ and $h''_1=s+{1\over 2}\zeta $. It is possible to cover the  interval $[s+{1\over
2}\zeta, \alpha_+ ]$  by a finite number of open intervals
 $\{(h'_i, h''_i)\}_{2\leq i\leq \ell} $   with length less than $\zeta$.   We have the inclusion
\begin{equation}
\label{decomp1}
\mathcal{E}_{>s}  \subset   \bigcup_{N=1}^{+\infty} \ \ \ \bigcap_{n\geq N} \ \    \bigcup_{i=1}^{\ell} \  \mathcal{Y}_n(h'_i, h''_i)  .
\end{equation}
Recall that $ I  _n(y)$ is the unique basic interval contained in  $\mathcal{C}_n$ containing $y$.
The embedding property (\ref{decomp1}) emphasizes that when
$\underline{d}_{\mup}(y)>s$ for a point $y\in \zu$, then
necessarily $\mup(I_n(y)) < |I_n(y)|^s$ for {\it every} integer $n$
large enough, not only for an infinite number of integers.

We introduce the subset   $ {\mathcal{C}}_{n,a}(x) $   of  $\mathcal{C}_n$
$$ {\mathcal{C}}_{ n,a}(x):=\left\{C\in \mathcal{C}_n: \ \tau(x,C) < 2^{an} \right\}.$$

By the definition of $\mathcal{R}_{\leq s}(x)$, for any real number $a>s$,
$\mathcal{R}_{\leq s}(x)  \subset \{y\in \zu: R(x,y)
<a\}$. The distorsion property \eqref{distor}  
guarantees that $I_n(y)$ tends to zero very regularly when $n$ tends
to infinity. Hence,  if $y \in \mathcal{R}_{\leq s}(x)$, there is an
infinite number of integers $n$ such that $ \tau_{2^{-n}}(x,y) \leq
2^{an}$, which means that $T^p x  \in B(y,2^{-n}) $ for some $p\leq
2^{an}$.


Denote by $d(y, C)$ the distance from  $y$ to the set $C$.
We introduce the sets
$$
\widetilde { \mathcal{Y}}_{n,a}(x) = \{ y\in \zu: \exists \, C\in
{\mathcal{C}}_{ n,a}(x)  \mbox{ such that } d(y,C) \leq  2^{-n} \}.
$$
 Since $T^p x  \in B(y,2^{-n}) $ implies that $d(y, I  _n(T^p x)) \leq
2^{-n}$, we have
 \begin{equation}
\label{decomp2}
\mathcal{R}_{\leq s}(x)   \subset \bigcap_{N=1}^{+\infty}   \ \ \bigcup_{n\geq N} \  \widetilde { \mathcal{Y}}_{n,a}(x).
\end{equation}
 Thus combining \eqref{decomp1} and \eqref{decomp2}, we get that $\displaystyle \mathcal{R}_{\leq s}(x)  \cap \mathcal{E}_{>s}$ is included in
\begin{eqnarray*}
\bigcap_{N=1}^{+\infty}   \ \ \bigcup_{n\geq N} \   \left(    \widetilde { \mathcal{Y}}_{n,a}(x) \cap
\bigcup_{i=1}^{\ell}\mathcal{Y}_n(h_i',
h_i'')\right) 
 \subset  \bigcup_{i=1}^{\ell}  \ \ \bigcap_{N=1}^{+\infty}    \ \ \bigcup_{n\geq N} \
\left( \widetilde { \mathcal{Y}}_{n,a}(x) \cap \mathcal{Y}_n(h_i', h_i'')\right).
\end{eqnarray*}
The above inversion of $\cap$ and $\cup$ follows from the fact that there is a finite number of intervals $[h'_i, h''_i]$.
Thus,  we need only to show that for all $1\leq i \leq \ell $,
$$
\forall \, \epsi>0, \ \ \ \dim   \big( \limsup_{n\to \infty}  \big( \widetilde { \mathcal{Y}}_{n,a}(x) \cap \mathcal{Y}_n(h_i',
h_i'')\big) \big) \le  D_{\mu_\phi}(s)+ \epsi .
$$

Let $\mathcal{C}_{n,a, h_i',h_i''}(x)$ be the subset of the basic
intervals of  $\mathcal{C}_n$  belonging  to  both $\mathcal{C}_{ n,a}(x) $ and $
\mathcal{C}_n(h_i', h_i'')$.  
\begin{lemma}
\label{lemfinal}
For every $a\in (s, \alpha_{\max})$, for every $\epsi>0$, for each $2\leq i
\leq \ell $,
\begin{eqnarray}\label{sum-proba}
   \sum_n \mu_\phi \big(\{x:  {\rm Card}\,  \mathcal{C}_{n,a, h_i',h_i''}(x)> 2^{n (D_{\mu_\phi} (a)+\epsi)}\} \big)<\infty.
\end{eqnarray}
\end{lemma}

 Observe that in Lemma \ref{lemfinal}, we do not consider the first interval $[h'_1,h''_1]$.  For this interval, \eqref{sum-proba}  follows from \eqref{cardinality1}, if   $\zeta$ is small enough. Indeed, if $\zeta<\!\!< 1$,     $h''_1= s+\zeta/2$ is very close to $s$, so that $D_{\mup}(h''_1)$ is close to $D_{\mup}(s)$.

Let us assume for a while that Lemma \ref{lemfinal} holds true. Then,   the Borel-Cantelli lemma yields that for $\mu_\phi$-almost every $x$, there exists an integer $n(x)$ such that 
$$\mbox{for $n\geq n(x)$,  } \ \ {\rm Card} \,   \mathcal{C}_{n,a, h_i',h_i''}(x)  \leq 2^{n (D_{\mu_\phi} (a)+\epsi)}.$$

In order to obtain a covering of the set   $\limsup_{n\to \infty}  \big( \widetilde { \mathcal{Y}}_{n,a}(x) \cap \mathcal{Y}_n(h_i',
h_i'')\big) $, by construction one   considers, for any $N\geq 1$,  the union
$$\bigcup_{n\ge N}  \ \    \bigcup_{C\in \mathcal{C}_{n,a, h_i',h_i''}(x)}  \ \{y\in \zu: d(y,C)\leq 2^{-n}\}.$$
Using this family of coverings, if $N\ge n(x)$, then for any $
\epsi>0$,  the $({D_{\mu_\phi}(a)+ 2\epsi})$-Hausdorff measure of the
above limsup set is bounded by
\begin{eqnarray*}
  \sum_{n\ge N} \,  \sum_{C\in  \mathcal{C}_{n,a, h_i',h_i''}(x)}  \!   \!   \!  
((L+2)\cdot 2^{-n}
    )^{D_{\mu_\phi}(a)+ 2 \epsi}  \!   &   \leq  &  \!    \Theta'  \ \sum_{n\ge N} 2^{-n (D_{\mu_\phi}(a)+ 2\epsi)} \cdot 2^{n( D_{\mu_\phi}(a) +\epsi)}\\
 \!  &    \leq &  \!    \Theta' \  \sum_{n\ge N} 2^{-n  \epsi}
<\infty, \vspace{-1mm}
\end{eqnarray*}
where $\Theta'$ is some constant depending on $L$, $a$, $\mup$ and
$\epsi$. Hence, letting $N$ tend to infinity, we see that the
$({D_{\mu_\phi}(a)+ 2\epsi})$-Hausdorff measure of  the limsup set
$\limsup_{n\to \infty}  \big( \widetilde { \mathcal{Y}}_{n,a}(x)
\cap \mathcal{Y}_n(h_i', h_i'')\big) $ is necessarily 0. This
implies that
$$
\dim  \Big( \limsup_{n\to \infty}  \big( \widetilde { \mathcal{Y}}_{n,a}(x) \cap \mathcal{Y}_n(h_i',
h_i'')\big) \Big) \le D_{\mu_\phi}(a) + 2\epsi.$$
We finish the proof of Proposition \ref{propfinal2}
 by letting
first $\e \downarrow 0$  and then $a \downarrow s$.

\sk


It remains us to prove  Lemma \ref{lemfinal}. Let $a \in (s, \alpha_{\max})$. It is enough to prove Lemma
\ref{lemfinal} for $a$ close to $s$. Hence we suppose that $a<h'_2$.

We assume that the intervals $[h'_i,h''_i]$ are chosen so that,
except for at most one of them, either $h'_i > \alpha_{\max}$ or
$h''_i <  \alpha_{\max}$. In other words, we suppose that there is
only one integer $i\in \{2,\ldots, l\}$ such that $\alpha_{\max} \in
(h'_i,h''_i)$.  We   use     two key properties:
\begin{itemize}

 \vspace{-1mm}\item
The multifractal spectrum $D_{\mup}$ is real-analytic and concave on
$]\alpha_-,\alpha_+[$.

 \vspace{-1mm}
\item
For every    $h\geq a >s> \dim  \mup$, the derivative of
$D_{\mup}$ at $s$ is strictly less than 1, and the derivative
$(D_{\mup})'(s)$ is decreasing. Hence,  there is  a real number
$0<\xi_a=(D_{\mup})'(a)<1$ such that  for every $h$ in every
interval $[h'_i,h''_i]$ ($i\geq 2$),
\begin{equation*}
\mbox{for every $h \geq a$, }  \ \ \ (D_{\mup})'(h)\leq \xi_a.
\end{equation*}
\end{itemize}
 
 \vspace{-5pt}

 We distinguish three cases.

\sk

$\bullet$ {\bf  If $h_i''< \alpha_{\max} $:} Take
$b=D_{\mu_\phi}(h_i'')+\epsi$, $  c=D_{\mu_\phi}(a)+\epsi$, and $
\eta=h_i'-a$. Then on the one hand, by (\ref{cardinality1}), for $n$
large enough there are at most $2^{bn}$ basic intervals $C$ in
$\mathcal{C}_n(h'_i,h''_i) =  \mathcal{C}_n(a+\eta,h''_i) $.
 On the other hand, by the mean value
theorem and the fact that $D_{\mu_\phi}(\cdot)$ is increasing on $(\alpha_-,
\alpha_{\max})$,
\begin{eqnarray*}
  b-c =D_{\mu_\phi}(h''_i)- D_{\mu_\phi}(a) <  \xi_a (h''_i-a) = \xi_a   (h'_i-a)+ \xi_a (h''_i-h'_i).
\end{eqnarray*}
Since $ \xi_a<1$ and $h''_i-h'_i<\zeta$, we can choose $\zeta$
small enough such that $$b-c<h'_i-a=\eta.$$ This choice of $\zeta$
can be uniform, i.e. valid for every index $i$ such that $h_i''<
\alpha_{\max} $.

By Lemma \ref{Small
hitting},  for sufficiently large $n$,
\begin{equation}
\label{final}
\mu_{\phi} \left( \left\{x: \,\begin{cases}  \ \ \tau(x,C) \le  2^{an} \text{ for } 2^{cn}
 \text{ distinct  }   \\ \text{  \ \ intervals  $C$ among the}\ 2^{bn} \text{
intervals}    \end{cases} \right\} \right)\le 2^{-  n}.
\end{equation}
This is equivalent to say that
\begin{align*}
\mu_\phi \big( \{x:  {\rm Card}  \, \mathcal{C}_{n,a, h_i',h_i''}(x)  > 2^{n(
D_{\mu_\phi}(a) + \epsi)}  \} \big) \le 2^{ - n}.
\end{align*}
Then (\ref{sum-proba}) follows.

\mk

$\bullet$ {\bf If $i$ is the unique integer such that $h_i'\leq
\alpha_{\max} \leq h_i'' $:} This occurs for one and only one
interval $[h'_i,h''_i]$. Recall that the cardinality of
$\mathcal{C}_n$ is less than $L \cdot 2^{n}$. Take $b=1$,
$c=D_{\mu_\phi}(a)$, and $ \eta=h_i'-a$. Then
\begin{eqnarray*}
  b-c  = 1- D_{\mu_\phi}(a) <  \xi_a(\alpha_{\max}-a) =  \xi_a(h'_i-a)+ \xi_a(\alpha_{\max}-h'_i).
\end{eqnarray*}
Since $D_{\mu_\phi}'(a)<1$ and $\alpha _{\max}-h'_i\le h''_i-h'_i<\zeta$, we can
choose $\zeta$ small enough   that $$b-c<h'_i-a=\eta.$$ By Lemma \ref{Small hitting} and applying the same arguments as above,  for
  large $n$, we have
$$
\mu_{\phi} \left( \left\{x: \,\begin{cases}  \ \ \tau(x,C) \le  2^{an} \text{ for } 2^{cn}
 \text{ distinct  }   \\ \text{  \ \ intervals  $C$ among  }\ 2^{bn} \text{
intervals of $\mathcal{C}_n$}    \end{cases} \right\} \right)\le 2^{-  n}.
$$
It is not difficult to prove that we can replace $2^{bn}$ by $L\cdot 2^{bn}$,  
since constants do not infer in the proofs of Lemma \ref{Small hitting}. In other words,
\begin{align*}
\mu_\phi \Big( \{x:  {\rm Card}  \, \mathcal{C}_{n,a, h_i',h_i''}(x)  > 2^{n
D_{\mu_\phi}(a)  }  \} \big) \le 2^{-  n},
\end{align*}
 and (\ref{sum-proba}) is proved.

\mk

$\bullet$ {\bf If  $ \alpha_{\max}<h_i'$:} Take $b=D_{\mu_\phi}(h_i')+\epsi$, $c=D_{\mu_\phi}(a)+\epsi$, and $
\eta=h_i'-a$. On the one hand, by (\ref{cardinality2}),  for $n$ large enough  there
are at most $2^{bn}$ basic intervals  in  $\mathcal{C}_n(h'_i,h''_i) =  \mathcal{C}_n(a+\eta,h''_i) $.
On the other hand,
\begin{eqnarray*}
  b-c = D_{\mu_\phi}(h'_i)- D_{\mu_\phi}(a) <  \xi_a (h'_i-a)<(h'_i-a).
\end{eqnarray*}
 Thus  by Lemma \ref{Small hitting},  for
sufficiently large $n$, \eqref{final} follows, and (\ref{sum-proba}) is proved.
\end{proof} 

\subsection{Conclusion.}
Combining Propositions  \ref{propfinal1} and  \ref{propfinal2}, we
have that for every $s\in (\dim  {\mup},\alpha_{\max})$,  
\begin{equation}
\mbox{for
$\mu_\phi$-almost every $x$,
  } \ \ \dim _H  \mathcal{R}_{\leq s} (x)= D_{\mup}(s).
\end{equation}

Then by Lemma \ref{lem-inclusion}, we have for every $\delta$ such
that $\dim  \mu_\phi< 1/\delta \leq \alpha_{\max}$, 
\begin{equation}\label{finfin}
\mbox{for
$\mu_\phi$-almost every $x$,
  } \ \ \dim _H  \Lkap= D_{\mup}(1/\delta).
\end{equation}

As we did in proving Part III and Part IV, by  noticing the
monotonicity of sets $\Lkap$ with respect to $\delta$ and applying
\eqref{finfin} to a dense countable set, we can obtain that for
$\mup$-almost every $x\in \zu$ for every $\delta$, (\ref{finfin})
holds. This concludes the proof.


\end{document}